\documentclass[12pt]{article}

\usepackage[english]{babel}
\usepackage[utf8x]{inputenc}
\usepackage[T1]{fontenc}

\newcommand\textbfit[1]{\textbf{\textit{#1}}}

\usepackage[a4paper,top=3cm,bottom=2cm,left=3cm,right=3cm,marginparwidth=1.75cm]{geometry}

\usepackage{amsmath,amsthm,amsfonts,amssymb}
\usepackage{graphicx}
\usepackage[colorinlistoftodos]{todonotes}
\usepackage[colorlinks=true,allcolors=blue]{hyperref}
\usepackage{tikz}
\usetikzlibrary{cd}
\usetikzlibrary{shapes.geometric,positioning}  
\usepackage{multicol}
\usepackage[overload]{abraces}
\newcommand\restrict[1]{\raisebox{-.5ex}{$|$}_{#1}}
\usepackage{rotating}
\usepackage{enumerate}
\usepackage{titling}
\newcommand{\subtitle}[1]{%
  \posttitle{%
  \vskip 1 em
    \par\end{center}
    \begin{center}\large#1\end{center}
   }%
}

\theoremstyle{plain}
\newtheorem{thm}{Theorem}[subsection]
\newtheorem{prop}[thm]{Proposition}
\newtheorem{lemma}[thm]{Lemma}
\newtheorem{exmp}[thm]{Example}

\theoremstyle{definition}

\newtheorem*{remark}{\normalfont{\textit{Remark}}}
\newtheorem{defn}[thm]{Definition} 
\makeatletter
\def\@cite#1#2{\textup{[{#1\if@tempswa , #2\fi}]}}
\makeatother

\title{A Categorical Approach to Subgroups of Quantum Groups and their Crystal Bases}
\subtitle{Extended Essay Submitted for Part B MMath Mathematics\\
The University of Oxford}
\author{Rhiannon Savage
}

\date{Hilary Term 2019}

\begin{document}
\maketitle

\thispagestyle{empty}

\begin{abstract}
\noindent Suppose that we have a semisimple, connected, simply connected algebraic group $G$ with corresponding Lie algebra $\mathfrak{g}$. There is a Hopf pairing between the  universal enveloping algebra $U(\mathfrak{g})$ and the coordinate ring $O(G)$. By introducing a parameter $q$, we can consider quantum deformations $U_q(\mathfrak{g})$ and $O_q(G)$ respectively, between which there again exists a Hopf pairing.  We show that the category of crystals associated with $U_q(\mathfrak{g})$ is a monoidal category. We define subgroups of $U_q(\mathfrak{g})$ to be right coideal subalgebras, and subgroups of $O_q(G)$ to be quotient left $O_q(G)$-module coalgebras. Furthermore, we discuss a categorical approach to subgroups of quantum groups which we hope will provide us with a link to crystal basis theory.
\end{abstract}

\pagebreak

\thispagestyle{empty}
    \clearpage\mbox{}\clearpage

\pagebreak

\tableofcontents
\thispagestyle{empty}
\bigskip
\subsection*{Notation}
\begin{itemize}
    \item We set $k$ to be an algebraically closed field of characteristic 0. Unless otherwise stated, algebraic structures are defined over $k$. The identity of $k$ is 1.
    \item The variable $q$ is a fixed non zero scalar, which is not a root of unity, in $k$. For a discussion of quantum groups at roots of unity see, for example, \cite{lusztig90}.
    \item We occasionally omit the structure maps for algebras, Hopf algebras etc. , and simply refer to them as $A$, $H$ etc. 
    \item $\text{id}_X$ denotes the identity map on $X$.
    \item We refer to (co)unital, (co)associative (co)algebras as `(co)algebras'.
\end{itemize}
\subsection*{Prerequisites}
Although this is a Part B Extended Essay, knowledge of material from the following Part C courses is assumed: Category Theory, Lie Algebras, Representation Theory of Semisimple Lie Algebras, Algebraic Geometry.

\thispagestyle{empty}
\pagebreak
\section{Introduction}
\setcounter{page}{1}

The term `quantum' is used throughout many areas of Physics and Mathematics and has many different interpretations. Quantum mechanics, for example, can be interpreted as a non-commutative deformation of classical mechanics by introducing Planck's constant $h$. As $h$ tends to 0 we, in some sense, get our classical mechanics back. By analogy, quantum groups can be thought of as non-commutative deformations of certain classical algebraic structures by introducing a parameter $q$. Similarly, as $q$ tends to 1, we recover our classical algebra. There is no precise definition of a `quantum group' and indeed seeking to define one is not particularly informative. Our main philosophy in this essay will be that quantum group theory is the study of quantum analogues of classical ideas.\\

In Section \ref{introtohopfalgebras}, we define the notions of algebras and coalgebras. We note that the dual of a coalgebra is an algebra. The converse to this statement is true for finite dimensional algebras, but for infinite dimensional algebras $A$ we instead consider the restricted dual $A^\circ$, which is a coalgebra. A Hopf algebra is both an algebra and a coalgebra, satisfying certain compatibility conditions, equipped with an antiautomorphism known as an antipode. The restricted dual of a Hopf algebra is a Hopf algebra. As algebraic objects, Hopf algebras can be thought of as generalisations of groups. The structure maps mimic the structure of a group, with the antipode generalising the inverse operation. This interpretation has led many Mathematicians to define quantum groups to be non-commutative, non-cocommutative Hopf algebras. \\

We focus on the quantum groups associated with certain algebraic groups $G$ and their associated Lie algebras $\mathfrak{g}$. In Section \ref{quantiseduniversalenvelopingalgebra}, we discuss how we can embed a Lie algebra $\mathfrak{g}$ into an algebra $U(\mathfrak{g})$, its universal enveloping algebra, containing all its representations. Furthermore, $U(\mathfrak{g})$ can be given a Hopf algebra structure. We present the quantised universal enveloping algebra $U_q(\mathfrak{g})$ of a finite dimensional complex semisimple Lie algebra $\mathfrak{g}$, which again has the structure of a Hopf algebra. This construction was first introduced by  Drinfield and Jimbo, independently, in their study of the quantum Yang-Baxter equation. We note that as $q$ tends to $1$ we, in some sense, reclaim our classical universal enveloping algebra. \\

Suppose that $G$ is a semisimple, connected, simply connected algebraic group with associated Lie algebra $\mathfrak{g}$. Then, the coordinate ring $O(G)$ is isomorphic to the restricted dual of $U(\mathfrak{g})$. In Section \ref{thequantisedcoordinatering}, we construct the quantised coordinate ring $O_q(G)$ of $G$ using the coordinate functions of certain $U_q(\mathfrak{g})$-modules. We see that $O_q(G)$ is a Hopf subalgebra of $U_q(\mathfrak{g})^\circ$, in analogy with the classical case. \\

The following diagram illustrates the links between the classical and quantum world discussed so far.
\begin{center}
\begin{tikzpicture}
\node(a) at (0,-1) {$U(\mathfrak{g})^\circ$};
\node(b) at (0,0) {$U(\mathfrak{g})^*$};
\node(c) at (6,-1) {$U_q(\mathfrak{g})^\circ$};
\node(e) at (6,0) {$U_q(\mathfrak{g})^*$};
\node(d) at (0,-2) {$O(G)$};
\node(f) at (6,-2) {$O_q(G)$};
\node(g) at (0,-0.5) {\begin{turn}{90} $\subseteq$
\end{turn}};
\node(h) at (6,-0.5) {\begin{turn}{90} $\subseteq$
\end{turn}};
\node(i) at (0,2) {$U(\mathfrak{g})$};
\node(j) at (6,2) {$U_q(\mathfrak{g})$};
\node(k) at (0,-1.5) {\begin{turn}{90} $\simeq$
\end{turn}};
\node(l) at (6,-1.5) {\begin{turn}{90} $\subseteq$
\end{turn}};
\draw [->] (d)--(f)  node[midway,above] {\textit{quantisation}};
\draw [->] (i)--(j) node[midway,above] {\textit{quantisation}} ;

\draw[->] (i) -- (b) node[midway,left] {\textit{dual}};
\draw[->] (j) --(e) node[midway,right] {\textit{dual}};;

\end{tikzpicture}    
\end{center}

An important question to ask is whether we can translate our classical concept of a subgroup into the language of quantum groups. Knowing that quantum groups have a Hopf algebra structure, it seems intuitive to define subgroups to be Hopf subalgebras of $U_q(\mathfrak{g})$ or, dually, quotient Hopf algebras of $O_q(G)$. In Section \ref{subgroupsofquantumgroups} we show that this definition is too restrictive, and instead define, with justification, right coideal subalgebras of $U_q(\mathfrak{g})$ and quotient left $O_q(G)$-module coalgebras to be subgroups. In the latter part of this section we find a number of examples of subgroups of $U_q(\mathfrak{sl}_2)$ and construct corresponding subgroups of $O_q(SL_2)$.  \\

In Section \ref{categoricalapproachtosubgroupsofquantumgroups}, we consider a categorical approach to subgroups of quantum groups. In the classical case, if $G$ is a finite group and $H$ is a subgroup of $G$, then the category of comodules of $O(G)$ is a module category over the category of comodules of $O(H)$. We see that we obtain a similar result in the quantum case. Category theory is a useful tool in connecting mathematical concepts so we hope that this category-theoretic characterisation will provide us with the tools to study crystal bases for our subgroups. \\

We explore Kashiwara's definition of a category of crystals \cite{kashiwara95} in Section \ref{thecategoryofcrystals}. This is a monoidal category which admits a combinatorial structure called a crystal graph. The main examples of objects in this category are crystal bases of integrable $U_q(\mathfrak{g})$-modules. The name crystal comes from the fact that these integrable representations of $U_q(\mathfrak{g})$ have bases at $q=0$, which Kashiwara refers to as crystallisation at absolute zero.   \\

In Section \ref{crystalbasesforsubgroupsofquantumgroups}, we discuss crystal bases for some subgroups of quantum groups. In particular, we discuss Kashiwara's construction of crystal bases for the Borel subalgebras $U_q^{\geq 0}(\mathfrak{g})$ and $U_q^{\leq 0}(\mathfrak{g})$. We also consider crystal bases for the set of invariants of $U_q^\pm(\mathfrak{g})$ on $O_q(G)$. We see that the constructions presented in this section cannot be generalised to other subgroups. \\

Finally, in Section \ref{furtherresearch}, we discuss ways that the ideas presented in this essay could be developed, including possible applications of our categorical characterisation of subgroups to finding their crystal bases.\\

\section{An Introduction to Hopf Algebras }\label{introtohopfalgebras}

\subsection{Algebras and Coalgebras}

We first define the concepts of algebras and coalgebras. A coalgebra can be obtained from an algebra by `reversing the arrows' so is, in some sense, dual to an algebra.

\begin{defn}\cite[p. 39]{kassel95}
An \textbfit{algebra} $(A,\mu,\eta)$ consists of a vector space $A$, along with two linear maps $\mu:A\otimes A\rightarrow{A}$, $\eta:k\rightarrow{A}$, known as the \textit{multiplication} and the \textit{unit} respectively, such that the following diagrams\\

\noindent Associativity
$$\begin{tikzcd}
A\otimes A\otimes A \arrow[r, "\mu\otimes \text{id}_A"] \arrow[d, "\text{id}_A\otimes\mu"]
& A\otimes A \arrow[d, "\mu"] \\
A\otimes A \arrow[r, "\mu"]
& A
\end{tikzcd}$$
Unit
$$\begin{tikzcd}
k\otimes A \arrow[r, "\eta\otimes \text{id}_A"] \arrow[dr,"\simeq"]
& A\otimes A \arrow[d, "\mu"] 
& A\otimes k \arrow [l,"\text{id}_A\otimes \eta" above] \arrow[dl, "\simeq" above]\\
& A
\end{tikzcd}$$
commute.
\end{defn}

\begin{defn}
\begin{enumerate}
    \item A subspace $B$ of $A$ is a \textbfit{subalgebra} if $\mu(B\otimes B)\subseteq B$ and $\eta(1)=1_A\in B$.
    \item Given two algebras $(A,\mu,\eta)$ and $(A',\mu',\eta')$, an \textbfit{algebra morphism} is a linear map $f:A\rightarrow{A'}$ such that
\begin{equation}\mu'\circ (f\otimes f)=f\circ\mu\text{ and } f\circ\eta=\eta'\end{equation}
\end{enumerate}

\end{defn}

\begin{defn}\cite[p. 40]{kassel95}
A \textbfit{coalgebra} $(C,\Delta, \varepsilon)$ consists of a vector space $C$, along with two linear maps $\Delta:C\rightarrow{C\otimes C}$, $\varepsilon:C\rightarrow{k}$, known as the \textit{comultiplication} and the \textit{counit} respectively, such that the following diagrams\\
\flushleft Coassociativity
$$\begin{tikzcd}
C \arrow[r, "\Delta"] \arrow[d, "\Delta"]
& C\otimes C \arrow[d, "\text{id}_C\otimes\Delta "] \\
C\otimes C \arrow[r, "\Delta\otimes \text{id}_C"]
& C\otimes C\otimes C
\end{tikzcd}$$
Counit
$$\begin{tikzcd}
k\otimes C \arrow[r,<-, "\varepsilon\otimes \text{id}_C"] \arrow[dr,<-,"\simeq"]
& C\otimes C \arrow[d,<-, "\Delta"] 
& C\otimes k \arrow [l,<-,"\text{id}_C\otimes \varepsilon" above] \arrow[dl,<-, "\simeq" above]\\
& C
\end{tikzcd}$$

commute.

\end{defn}

\begin{remark}
If $c$ is an element of a coalgebra $(C,\Delta, \varepsilon)$ then, for $c_{1,i},c_{2,i}\in C$, we have $\Delta(c)=\sum_i c_{(1),i}\otimes c_{(2),i}=\sum c_{(1)}\otimes c_{(2)} $ in `\textit{Sweedler notation}'. The subscripts `$(1)$' and `$(2)$' indicate the order of the factors in the tensor product. Using the counit diagram above, $c=\sum\varepsilon(c_{(1)})c_{(2)}=\sum c_{(1)}\varepsilon(c_{(2)}).$
\end{remark}

\begin{defn}
\begin{enumerate}
    \item A subspace $D$ of $C$ is a \textbfit{subcoalgebra} if $\Delta(D)\subseteq D\otimes D$.
    \item Given two coalgebras $(C,\Delta, \varepsilon)$ and $(C',\Delta', \varepsilon')$, a \textbfit{coalgebra morphism} is a linear map $f:C\rightarrow{C'}$ such that
    \begin{equation}(f\otimes f)\circ\Delta=\Delta'\circ f\text{ and } \varepsilon=\varepsilon'\circ f\end{equation}
\end{enumerate}
\end{defn}
Algebras act on modules, and therefore it seems natural to consider coalgebras coacting on comodules. We first define our notion of a module over an algebra.
\begin{defn}\cite[p. 61]{kassel95}
Suppose $(A,\mu,\eta)$ is an algebra. A \textbfit{left $A$-module} $(M,\mu_M)$ consists of a vector space $M$, along with a linear map $\mu_M:A\otimes M\rightarrow{M}$, the \textit{action} of $A$ on $M$, such that the following diagrams\\

\flushleft Associativity
$$\begin{tikzcd}
A\otimes A\otimes M \arrow[r, "\mu\otimes \text{id}_M"] \arrow[d, "\text{id}_A\otimes \mu_M"]
& A\otimes M \arrow[d, "\mu_M "] \\
A\otimes M \arrow[r, "\mu_M"]
& M
\end{tikzcd}$$
Unit
$$\begin{tikzcd}
k\otimes M\arrow[r, "\eta\otimes \text{id}_M"] \arrow[dr,"\simeq"]
& A\otimes M \arrow[d, "\mu_M"]\\ 
& M
\end{tikzcd}$$

\flushleft commute.

\end{defn}

\begin{defn}\cite[p. 62]{kassel95}
Suppose $(C,\Delta, \varepsilon)$ is a coalgebra. A \textbfit{right C-comodule} $(N,\Delta_N)$ consists of a vector space $N$, along with a linear map $\Delta_N:N\rightarrow{N\otimes C}$, the \textit{coaction} of $C$ on $N$, such that the following diagrams\\

\flushleft Coassociativity
$$\begin{tikzcd}
N \arrow[r, "\Delta_N"] \arrow[d, "\Delta_N"]
& N\otimes C \arrow[d, "\text{id}_N\otimes \Delta "] \\
N\otimes C \arrow[r, "\Delta_N\otimes \text{id}_C"]
& N\otimes C\otimes C
\end{tikzcd}$$
Counit
$$\begin{tikzcd}
N\otimes k\arrow[r,<-, "\text{id}_N\otimes\varepsilon"] \arrow[dr,<-,"\simeq"]
& N\otimes C \arrow[d,<-, "\Delta_N"]\\ 
& N
\end{tikzcd}$$

\flushleft commute.

\end{defn}

\begin{remark}
If $n$ is an element of a right comodule $(N,\Delta_N)$, then $\Delta_N(n)\in N\otimes C$ is of the form $\sum n_{(0)}\otimes n_{(1)}$ in Sweedler Notation. We can define left $C$-comodules similarly. Similarly, if $n'$ is an element of a left comodule $(N',\Delta_{N'})$ then $\Delta_{N'}(n')=\sum n'_{(-1)}\otimes n'_{(0)}$.
\end{remark}

\subsection{Duality of Algebras and Coalgebras}\label{dualityalgcoalg}\label{dualalgcoalg}

We now explore the duality between coalgebras and algebras. Suppose that we have a vector space $V$; by definition of the tensor product of linear maps, detailed in \cite[Chapter II]{kassel95}, we have
\begin{equation}(f_1\otimes f_2)(v_1\otimes v_2)=f_1(v_1)\otimes f_2(v_2)\end{equation}
for $f_1,f_2\in V^*$ and $v_1,v_1\in V$. Therefore, $V^*\otimes V^*$ can be identified with a subset of $(V\otimes V)^*$. Moreover, when $V$ is finite dimensional $\text{dim}(V^*\otimes V^*)=\text{dim}((V\otimes V)^*)$ so $V^*\otimes V^*=(V\otimes V)^*$.

\begin{prop}\cite[Lemma 1.2.2]{montgomery93}
The dual of a coalgebra is an algebra.
\end{prop}
\begin{proof}
Suppose that we have a coalgebra $(C,\Delta,\varepsilon)$. The transpose of $\Delta$, \\$\Delta^*:(C\otimes C)^*\rightarrow{C^*}$, restricts to a map $\Delta^*\restrict{C^*\otimes C^*}:C^*\otimes C^*\rightarrow{C^*}$. If we define $A=C^*$, $\mu=\Delta^*\restrict{C^*\otimes C^*}$, $\eta=\varepsilon^*$ then $(A,\mu,\eta)$ can be shown to be an algebra by checking the commutativity conditions. 
\end{proof}

\begin{remark}
The multiplication $\Delta^*\restrict{C^*\otimes C^*}$ is the \textit{convolution product} and is often written $f\star g$ for $f,g\in C^* $. If $c\in C$, 
\begin{equation}
\Delta^*\restrict{C^*\otimes C^*}(f\otimes g)(c)=(f\star g)(c)=(f\otimes g)\Delta(c)=\sum f(c_{(1)})g(c_{(2)})    
\end{equation}
\end{remark}

We want to consider under which conditions the dual of an algebra is a coalgebra. If $A$ is not a finite-dimensional algebra, $A^*\otimes A^*$ is a proper subset of  $(A\otimes A)^*$ and so the image of $\mu^*:A^*\rightarrow{(A\otimes A)^*}$ may not lie in $A^*\otimes A^*$. We define the restricted dual.

\begin{defn}\cite[Definition 1.2.3]{montgomery93}\label{rd}
The \textbfit{restricted dual} of an algebra $A$ is \begin{equation}A^\circ =\{f\in A^*: f(I)=0 \text{ for some ideal } I \text{ of } A \text{ with } dim_k(A/I)< \infty\}\end{equation}
\end{defn}

The following Lemma provides us with an alternative definition.

\begin{lemma}\cite[Lemma 1.4.8, Proposition 1.4.10]{timmermann08}\label{restricteddual}
$A^\circ$ is equivalent to 
\begin{equation}
\{f\in A^*\mid \mu^*(f)\in A^*\otimes A^*\}    
\end{equation}Moreover, $\mu^*(A^\circ)\subseteq A^\circ\otimes A^\circ$.
\end{lemma}

\begin{prop}
\cite[Proposition 1.2.4]{montgomery93} The restricted dual of an algebra is a coalgebra.
\end{prop}
\begin{proof}
Suppose that we have an algebra $(A,\mu,\eta)$. The map $\mu^*\restrict{A^\circ}:A^\circ\rightarrow{(A\otimes A)^*}$ is equivalent to a map $\mu^*\restrict{A^\circ}:A^\circ\rightarrow{A^\circ\otimes A^\circ}$ by Lemma \ref{restricteddual}. If we let $C=A^\circ$, $\Delta=\mu^*\restrict{A^\circ}$ and $\varepsilon=\eta^*\restrict{A^\circ}$, then  $(C,\Delta,\varepsilon)$ can be shown to be a coalgebra by checking the commutativity conditions.
\end{proof}
\begin{remark}
If $A$ is finite-dimensional then $A^\circ=A^*$, and so $A^*$ is a coalgebra.
\end{remark}
Moreover, we have the following duality between modules and comodules.

\begin{lemma}\cite[Section I.9.16]{browngoodearl02}\label{modcomod}
Let $A$ be an algebra, and $C$ be a coalgebra. 
\begin{enumerate}
    \item \label{part1} Suppose that $(N,\Delta_N)$ is a right $C$-comodule with $\Delta_N(n)=\sum n_{(0)}\otimes n_{(1)}$ for $n\in N$, then $N$ is a left $C^*$-module with module action
    \begin{equation} f\cdot n=\sum f(n_{(1)})n_{(0)}\end{equation}
    for $f\in C^*$.
    \item If $M$ is a left $A$-module, then $M$ can be made into a right $A^\circ$-comodule whose associated left $A$-module (as in \ref{part1}) is $M$ if and only if dim$_k(A\cdot m)<\infty$ for all $m\in M$.
\end{enumerate}
\end{lemma}

\subsection{Hopf Algebras}

\begin{defn}\cite[Definition I.9.7]{browngoodearl02}
A \textbfit{bialgebra} $(H,\mu,\eta,\Delta,\varepsilon)$ consists of a vector space $H$ such that:
\begin{itemize}
    \item $(H,\mu,\eta)$ is an algebra and $(H,\Delta,\varepsilon)$ is a coalgebra,
    \item The multiplication $\mu$ and the unit $\eta$ are coalgebra morphisms.
\end{itemize}
\end{defn} 
\begin{remark}
The condition that $\mu$ and $\eta$ are coalgebra morphisms is equivalent to $\Delta$ and $\varepsilon$ being algebra morphisms \cite[Theorem III.2.1]{kassel95}.
\end{remark}

We are now ready to make our definition of a Hopf algebra.
\begin{defn}\cite[Definition I.9.9]{browngoodearl02}\label{hopf}
A \textbfit{Hopf algebra} $(H,\mu,\eta,\Delta,\varepsilon, S)$ is a bialgebra equipped with a map $S:H\rightarrow{H}$, called the \textit{antipode}, such that 
\begin{equation}
    \mu\circ(S\otimes \text{id}_H)\circ\Delta=\mu\circ(\text{id}_H\otimes S)\circ\Delta=\eta\otimes\varepsilon
\end{equation}
We see that $S$ is an inverse to $\text{id}_H$ under the convolution product $\star$
\begin{equation}
    S\star \text{id}_H=\text{id}_H\star S=\eta\circ\varepsilon
\end{equation}
\end{defn}

We note the following definitions.
\begin{defn}\label{hopfsubalgebra}\begin{enumerate}
\item A \textbfit{(co)module} over a Hopf algebra is a (co)module over the underlying (co)algebra structure. 
\item A subset $H'$ of a Hopf algebra $H$ is a \textbfit{Hopf subalgebra} if it is both a subalgebra and a subcoalgebra of $H$, and $S(H')\subseteq H'$
\item \label{hopfmorphism} Given Hopf algebras $(H,\mu,\eta,\Delta,\varepsilon, S)$ and $(H',\mu',\eta',\Delta',\varepsilon', S')$ a \textbfit{Hopf algebra morphism} $f:H\rightarrow{H'}$ is both an algebra and a coalgebra morphism satisfying $f\circ S=S'\circ f$.
\end{enumerate}
\end{defn}
\begin{exmp}\label{exmp1}
The group algebra $kG$ has Hopf algebra structure with maps given on the basis $\{v_g\mid g\in G\}$ by
\begin{equation}\Delta(v_g)=v_g\otimes v_g,\quad \varepsilon(v_g)=1, \quad S(v_g)=v_g^{-1}\end{equation}
Similarly, the algebra of functions on $G$, $k(G)$, has the structure of a Hopf algebra with maps defined by
    \begin{equation}\label{algfunc}\Delta(f)(g_1, g_2)=f(g_1g_2),\quad \varepsilon(f)=f(e),\quad (S(f))(g)=f(g^{-1})\end{equation}
    for $f\in k(G)$ and $g,g_1,g_2\in G$, with $e$ the identity element of $G$.
\end{exmp}

\subsection{Duality of Hopf Algebras}\label{hopfdual}

Recall the duality between algebras and coalgebras given in Section \ref{dualityalgcoalg}; we obtain similar results for Hopf algebras. We define the restricted dual $H^\circ$ of a Hopf algebra $H$ in the same way as in Definition \ref{rd}. By \cite[Theorem 9.1.3]{montgomery93}, $H^\circ$ has a Hopf algebra structure
\begin{equation}(H^\circ,\mu_{H^\circ},\eta_{H^\circ},\Delta_{H^\circ},\varepsilon_{H^\circ},S_{H^\circ})=(H^\circ,\Delta^*\restrict{H^\circ\otimes H^\circ},\varepsilon^*,\mu^*\restrict{H^\circ},\eta^*\restrict{H^\circ}, S^*\restrict{H^\circ})\end{equation}
Again, we see that the dual of a finite dimensional Hopf algebra is a Hopf algebra.\\ 

We can extend the concept of a dual pair of vector spaces to a dual pairing of Hopf algebras.

\begin{defn} \cite[Definition I.9.22]{browngoodearl02}
A \textbfit{Hopf pairing} $(\cdot\,,\cdot):H\times U\rightarrow{k}$ of two Hopf algebras $(H,\mu_H,\eta_H,\Delta_H,\varepsilon_H, S_H)$ and $(U,\mu_U,\eta_U,\Delta_U,\varepsilon_U, S_U)$ is a bilinear form with:
\begin{itemize}
    \item $(h,uv)=\sum(h_{(1)},u)(h_{(2)},v)$,
    \item $(fh,u)=\sum(f,u_{(1)})(h,u_{(2)}),$ 
    \item $(1,u)=\varepsilon_U(u)$ and $(h,1)=\varepsilon_H(h),$
    \item $(h,S_Uu)=(S_Hh,u)$
\end{itemize}
for all $u,v\in U$ and $f,h\in H$.
\end{defn}

Suppose that $(\cdot\,,\cdot):H\times U\rightarrow{k}$ is a Hopf pairing. We define the following maps for all $u\in U,h\in H$
\begin{equation}\varphi:U\rightarrow{H^\circ}\subseteq H^*,\quad \phi(u)(h)=(h,u)\end{equation}
\begin{equation}\psi:H\rightarrow{U^\circ}\subseteq U^*,\quad \psi(h)(u)=(h,u)\end{equation}
These maps are Hopf algebra morphisms. If $\varphi,\psi$ are injective then we say that the Hopf pairing is \textit{perfect}. Therefore, if $H$ is a finite dimensional Hopf algebra then the pairing between $H$ and $H^*$ given by $(\cdot,\cdot):H\times H^*:\rightarrow{
k},(h,f)=f(h)$ is a perfect Hopf pairing. Further, if there exists a perfect Hopf pairing between finite dimensional Hopf algebras $U$ and $H$ then $U\simeq H^*$ and $H\simeq U^*$ as Hopf algebras.

\begin{exmp}\label{example2}
We can define a bilinear form $(\cdot,\cdot):kG\times k(G)\rightarrow{k}$ by
\begin{equation}(\sum _{g\in G}a_gv_g,f)=\sum_{g\in G}a_gf(v_g)\end{equation} This is a Hopf pairing which induces an isomorphism of Hopf algebras $\text{Rep}(G)\simeq (kG)^\circ$, where $\text{Rep}(G)$ is the Hopf algebra of representative functions on $G$, see \cite[Example 1.4.12]{timmermann08}
\end{exmp}

Examples \ref{exmp1} and \ref{example2} provide a good starting point for our explorations in the next few sections. Analogous to $k(G)$, we have the ring of polynomial functions $O(G)$ on an algebraic group $G$. If we consider the corresponding Lie algebra $\mathfrak{g}$, then the universal enveloping algebra $U(\mathfrak{g})$ has parallels with $kG$; both have a universal property and translate representation theory into module theory. We will see that $O(G)$ and $U(\mathfrak{g})$ both have Hopf algebra structures and moreover, there is an isomorphism of Hopf algebras $O(G)\simeq U(\mathfrak{g})^\circ$.

\section{The Quantised Universal Enveloping Algebra}\label{quantiseduniversalenvelopingalgebra}

\subsection{The Universal Enveloping Algebra} \label{unienvalgsection}

We first recall the familiar definition of a Lie algebra.
\begin{defn}\cite[p. 1 - 2]{hongkang02}
A \textbfit{Lie algebra} $\mathfrak{g}$ is a vector space $\mathfrak{g}$ equipped with a bilinear map, the \textit{Lie bracket}, such that for all $x,y,z\in \mathfrak{g}$ we have:
\begin{itemize}
    \item (Anticommutativity) $[x,y]=-[y,x]$,
    \item (The Jacobi Identity) $[x,[y,z]]+[y,[z,x]]+[z,[x,y]]=0$.
\end{itemize}
\end{defn}

\begin{exmp}\label{sl2} \cite[p. 99]{kassel95} We denote by $\mathfrak{gl}_2$ the \textbfit{general linear Lie algebra of order $2$}. It has a basis 
\begin{equation}I=\begin{pmatrix} 1 & 0 \\ 0 & 1\end{pmatrix},
\quad 
e=\begin{pmatrix} 0 & 1 \\ 0 & 0\end{pmatrix},
\quad f=\begin{pmatrix} 0 & 0 \\ 1 & 0\end{pmatrix},
\quad h=\begin{pmatrix} 1 & 0 \\ 0 & -1\end{pmatrix}\end{equation}
The Lie bracket is the commutator $[x,y]=xy-yx$ for $x,y\in\mathfrak{gl}_2$. The basis elements satsify the following relations 
\begin{equation}
[e,f]=h,\quad [h,e]=2e, \quad [h,f]=-2f    
\end{equation}
\begin{equation}
[I,e]=[I,f]=[I,h]=0    
\end{equation}

The matrices of trace zero in $\mathfrak{gl}_2$ form a Lie subalgebra with basis $\{e,f,h\}$. This subalgebra is the \textbfit{special linear Lie algebra of order $2$}, $\mathfrak{sl}_2$ . It is easy to show, using the above relations, that there is an isomorphism of Lie algebras
\begin{equation}\mathfrak{gl}_2\simeq\mathfrak{sl}_2\otimes kI\end{equation}
This means that we can investigate many properties of $\mathfrak{gl}_2$ by just studying those of $\mathfrak{sl}_2$. The main examples in this essay will concern $\mathfrak{sl}_2$.
\end{exmp}

The universal enveloping algebra of a Lie algebra is the largest algebra containing all representations of a Lie algebra. The idea is to embed $\mathfrak{g}$ into an algebra $U(\mathfrak{g})$ such that the Lie bracket of $\mathfrak{g}$ corresponds to the commutator $xy-yx$ in $U(\mathfrak{g})$. 

We make a few preliminary definitions. Let $T(\mathfrak{g})$ be the \textit{tensor algebra} of $\mathfrak{g}$ defined by
\begin{equation}
    T(\mathfrak{g})=\sum_{n\geq 0}T^n(\mathfrak{g})
\end{equation}
where $T^0(\mathfrak{g})=k, T^1(\mathfrak{g})=\mathfrak{g}$ and $T^n(\mathfrak{g})=V^{\otimes n}$, and multiplication given by the tensor product 
\begin{equation}
    T^i(\mathfrak{g)}\otimes T^j(\mathfrak{g})\rightarrow{T^{i+j}(\mathfrak{g})}
\end{equation}
Let $I(\mathfrak{g})$ be the two sided ideal of $T(\mathfrak{g})$ generated by all elements of the form
\begin{equation}
x\otimes y-y\otimes x-[x,y]    
\end{equation} 
for $x,y\in\mathfrak{g}$ .

\begin{defn}\cite[Definition 1.7]{ciubotaru18}
The \textbfit{universal enveloping algebra} $(U(\mathfrak{g}),i_\mathfrak{g})$ of a Lie algebra $\mathfrak{g}$ consists of an algebra $U(\mathfrak{g})$ and a linear map $i_\mathfrak{g}:\mathfrak{g}\rightarrow{U(\mathfrak{g})}$. We define
\begin{equation}
    U(\mathfrak{g})=T(\mathfrak{g})/I(\mathfrak{g})
\end{equation}
The map $i_\mathfrak{g}$ is obtained by composing the identity map $\mathfrak{g}\rightarrow{T^1(\mathfrak{g})}$ with the quotient map $T(\mathfrak{g})\rightarrow{U(\mathfrak{g})}$.
\end{defn}

This algebra satisfies a universal property, motivating its name. 
\begin{thm}\cite[Lemma 1.8]{ciubotaru18}
Given an algebra $A$ together with a linear map $\tau:\mathfrak{g}\rightarrow{A}$ satisfying
\begin{equation}
    \tau(x)\tau(y)-\tau(y)\tau(x)=\tau([x,y])
\end{equation}
for all $x,y\in\mathfrak{g}$, there exists a unique algebra morphism $\phi:U(\mathfrak{g})\rightarrow{A}$ such that $\phi(1_{U(\mathfrak{g})})=1_A$ and
\begin{equation}
    \phi\circ i_\mathfrak{g}=\tau
\end{equation}
\end{thm}

Moreover, $U(\mathfrak{g})$ provides us with a link between Lie algebras and Hopf algebras.

\begin{prop}\cite[Section 1.2.9]{timmermann08}
 $U(\mathfrak{g})$ has the structure of a Hopf algebra with structure maps satisfying
\begin{align}
\begin{aligned}
&\Delta(x)& &=x\otimes 1+1\otimes x,\\
& \varepsilon(x)& &=0,\\
& S(x)& &=-x
\end{aligned}
\end{align} for $x\in \mathfrak{g}$
\end{prop}
\begin{proof}
We consider the linear maps 
\begin{align}
\begin{aligned}
&\Delta :\mathfrak{g}\rightarrow{U(\mathfrak{g})\otimes U(\mathfrak{g})}, \quad &  &\varepsilon :\mathfrak{g}\rightarrow{k},\quad  &  &S :\mathfrak{g}\rightarrow{\mathfrak{g}}, \\
&\Delta(x)=x\otimes 1+1\otimes x,\quad  & &\varepsilon(x)=0,\quad  & &S(x)=-x,
\end{aligned}
\end{align}
These maps satisfy

\begin{equation}
\begin{split}
[\Delta(x),\Delta(y)] &=[x\otimes 1+1\otimes x,y\otimes 1+1\otimes y]=[x,y]\otimes 1+1\otimes [x,y]=\Delta([x,y]),\\[18pt]
[\varepsilon(x),\varepsilon(y)] &= 0 =\varepsilon([x,y]),\\[18pt]
[S(x),S(y)] &=[-x,-y]^{op}=yx-xy=-[x,y]=S([x,y])
\end{split}
\end{equation}

for all $x,y\in\mathfrak{g}$. By the universal property, these can be extended to the algebra morphisms
\begin{equation}\Delta:U(\mathfrak{g})\rightarrow{U(\mathfrak{g})\otimes U(\mathfrak{g})},\quad \varepsilon:U(\mathfrak{g})\rightarrow{k},\quad S:U(\mathfrak{g})\rightarrow{U(\mathfrak{g})^{op}}\end{equation}
We can then check that these maps satisfy the relations in Definition \ref{hopf}.
\end{proof}

In \cite{ciubotaru18}, a finite dimensional Lie algebra is defined to be \textit{semisimple} if it has zero radical. Suppose that we have a finite dimensional complex semisimple Lie algebra $\mathfrak{g}$. Let $\Pi\subseteq\Phi^+$ be the set of simple roots forming a basis of the root system $\Phi$ with respect to a fixed Cartan subalgebra $\mathfrak{h}$ of $\mathfrak{g}$. The Cartan decomposition of $\mathfrak{g}$ is:
\begin{equation}\label{cartandecomp}
    \mathfrak{g}=\mathfrak{h}\bigoplus_{\alpha\in\Phi}\mathfrak{g}_\alpha, \quad \text{ where } \mathfrak{g}_\alpha=\{x\in\mathfrak{g}\mid [h,x]=\alpha(h)x,\quad h\in\mathfrak{h}\}
\end{equation}

We can associate $\mathfrak{g}$ with a matrix.

\begin{defn}
The \textbfit{Cartan Matrix} of $\mathfrak{g}$ is the matrix $C=(a_{\alpha\beta})$ with
\begin{equation}
a_{\alpha\beta}=\frac{2(\alpha,\beta)}{(\alpha,\alpha)}    
\end{equation}
for all $\alpha,\beta\in\Pi$. Here, the map $(\cdot,\cdot):\mathfrak{g}\times\mathfrak{g}\rightarrow{\mathbb{C}}$ is defined by $(x,y)=\textrm{tr}(\text{ad}x\circ\text{ad}y)$, where $\textrm{ad}x:y\rightarrow{[x,y]}$ for $x,y\in\mathfrak{g}$, and is called the \textit{Killing form} of $\mathfrak{g}$.
\end{defn}

We give Serre's presentation of the universal enveloping algebra $U(\mathfrak{g})$ of $\mathfrak{g}$ over $\mathbb{C}$, which will motivate our definition of $U_q(\mathfrak{g})$ in Section \ref{quantunienvelsection}.
\begin{prop}\cite[p. 44]{browngoodearl02} \cite[Appendix]{serre87} \label{unienvalgserre} $U(\mathfrak{g})$ is the $\mathbb{C}$-algebra with generators $\{e_\alpha,f_\alpha,h_\alpha\}_{\alpha\in\Pi}$, satisfying the relations
\begin{equation}[h_\alpha,h_\beta]=0,\quad [e_\alpha,f_\beta]=\delta_{\alpha\beta}h_\alpha,\quad [h_\alpha,e_\beta]=a_{\alpha\beta}e_\beta,\quad [h_\alpha,f_\beta]=-a_{\alpha\beta}f_\beta\end{equation}
and for $\alpha\neq \beta$, the Serre Relations, 
\begin{equation}
    \begin{split}
\sum_{k=0}^{1-a_{\alpha\beta}}(-1)^k\binom{1-a_{\alpha\beta}}{k}e_\alpha^{1-a_{\alpha\beta}-k}e_\beta e_\alpha^k&=0 \\
\sum_{k=0}^{1-a_{\alpha\beta}}(-1)^k\binom{1-a_{\alpha\beta}}{k}f_\alpha^{1-a_{\alpha\beta}-k}f_\beta f_\alpha^k&=0
    \end{split}
\end{equation}
\end{prop}

Let $U^+(\mathfrak{g}),U^0(\mathfrak{g}),U^-(\mathfrak{g})$ be the subalgebras of $U(\mathfrak{g})$ generated by $\{e_\alpha\}_{\alpha\in\Pi}$,  $\{h_\alpha\}_{\alpha\in\Pi}$, $\{f_\alpha\}_{\alpha\in\Pi}$ respectively. We state the following result.
\begin{prop}\cite[Proposition 2.1.7]{hongkang02}\label{triangulardecomposition}
$U(\mathfrak{g})$ has the \textit{triangular decomposition} 
\begin{equation}
    U(\mathfrak{g})\simeq U^+(\mathfrak{g})\otimes U^0(\mathfrak{g})\otimes U^-(\mathfrak{g})
\end{equation}
\end{prop}
\begin{defn}\label{borelsub}
The \textbfit{positive Borel subalgebra} $U^{\geq 0}(\mathfrak{g})$ is the subalgebra of $U_q(\mathfrak{g})$ generated by $\{e_\alpha,h_\alpha\}_{\alpha\in\Pi}$. The \textbfit{negative Borel subalgebra} $U^{\leq 0}(\mathfrak{g})$ is generated by $\{f_\alpha,h_\alpha\}_{\alpha\in\Pi}$.
\end{defn}
\begin{exmp}\label{usl2}
We now consider the case when $\mathfrak{g}=\mathfrak{sl}_2(\mathbb{C})$, as in Example \ref{sl2}. With respect to the Cartan subalgebra $\mathfrak{h}=\mathbb{C}h$, the simple root is $\alpha:\mathfrak{h}\rightarrow{\mathbb{C}}$ with $\alpha(h)=2$. The Cartan matrix is $(2)$.  
$U(\mathfrak{sl}_2)$ is the $\mathbb{C}$-algebra generated by $e,f,h$ satisfying the relations 
\begin{equation}[e,f]=h,\quad [h,e]=2e,\quad [h,f]=-2f\end{equation}

\end{exmp}
\begin{remark}
We note that the subalgebra of $U(\mathfrak{g})$ generated by $\{e_\alpha,f_\alpha,h_\alpha\}$ is isomorphic to $U(\mathfrak{sl}_2)$.
\end{remark}

\subsection{The Quantised Universal Enveloping Algebra}\label{quantunienvelsection}

In Proposition \ref{unienvalgserre}, we gave a presentation for $U(\mathfrak{g})$ over $\mathbb{C}$. We can, however, define its `quantisation' $U_q(\mathfrak{g})$ over our field $k$ since the construction is only based on the Cartan matrix $(a_{\alpha\beta})$ of $\mathfrak{g}$. Let $q_\alpha=q^{\frac{(\alpha,\alpha)}{2}}$ for $\alpha\in\Pi$ and assume that $q_\alpha\neq \pm 1$. 

\begin{defn} \cite[Section I.6.3]{browngoodearl02}\label{quantenvalg}
The \textbfit{quantised universal enveloping algebra} $U_q(\mathfrak{g})$ is defined to be the algebra generated by $\{E_\alpha,F_\alpha,K_\alpha,K_\alpha^{-1}\}_{\alpha\in\Pi}$ satisfying the relations
\begin{align}
\begin{aligned}
K_\alpha K_\alpha^{-1}&=K_\alpha^{-1}K_\alpha=1, &\quad K_\alpha K_\beta &=K_\beta K_\alpha, \\
K_\alpha E_\beta K_\alpha^{-1}&=q_\alpha^{a_{\alpha \beta }}E_\beta ,&\quad K_\alpha F_\beta K_\alpha ^{-1}&=q_\alpha ^{-a_{\alpha\beta }}F_\beta, 
\end{aligned}
\end{align}
\begin{equation*}
E_\alpha F_\beta -F_\beta E_\alpha =\delta_{\alpha \beta }\frac{K_\alpha -K_\alpha ^{-1}}{q_\alpha -q_\alpha ^{-1}}    
\end{equation*}
And, for $\alpha \neq \beta $, the \textit{quantum Serre relations}
\begin{equation}
    \begin{split}
     \sum_{k=0}^{1-a_{\alpha \beta }}(-1)^k
    { {1-a_{\alpha \beta }}\brack k}_{q_\alpha }E_\alpha ^{1-a_{\alpha \beta }-k}E_\beta E_\alpha ^k&=0,\\
    \sum_{k=0}^{1-a_{\alpha \beta }}(-1)^k
    { {1-a_{\alpha \beta }}\brack k}_{q_\alpha }F_\alpha ^{1-a_{\alpha\beta }-k}F_\beta F_\alpha ^k&=0 
    \end{split}
\end{equation}
where
${ m\brack n}_{q_\alpha }=\frac{[m]_{q_\alpha }!}{[n]_{q_\alpha }![m-n]_{q_\alpha }!}$, with $[m]_{q_\alpha }!=\prod_{k=1}^m[k]_{q_\alpha }$, $[k]_{q_\alpha }=\frac{q_\alpha ^k-q_\alpha ^{-k}}{q_\alpha -q_\alpha ^{-1}}$.

\end{defn}

If we compare this definition with the classical case, we see that the classical Serre relations are very similar to the quantum Serre relations with binomial coefficients replaced with `$q$-binomial coefficients' ${ m\brack n}_{q_\alpha }$. Following our remarks in the Introduction, we see that we have introduced a parameter $q$ which has introduced some non-commutativity and we would expect that as $q\rightarrow{1}$ we get our classical algebra back. We illustrate this in the following example.

\begin{exmp}\label{Uq(sl2)}
We now return to the case of $\mathfrak{g=sl_2}$. Since $(a_{\alpha\beta})=(2)$, it follows that $U_q(\mathfrak{sl_2})$ is the algebra generated by $E,F,K,K^{-1}$ subject to the relations
\begin{equation}KK^{-1}=K^{-1}K=1,\quad KEK^{-1}=q^2E,\quad KFK^{-1}=q^{-2}F\end{equation}
and, \begin{equation}EF-FE=\frac{K-K^{-1}}{q-q^{-1}}\end{equation}
We would expect there to be a connection between $U_q(\mathfrak{sl}_2)$ and $U(\mathfrak{sl}_2)$ on taking the limit as $q\rightarrow{1}$. Denote by $U_q'(\mathfrak{sl}_2)$ the following algebra generated by the variables $E,F,K,K^{-1},L$ subject to the following relations
\begin{align}
\begin{aligned}
&KK^{-1}=K^{-1}K=1,\quad KEK^{-1}=q^2E,\quad KFK^{-1}=q^{-2}F\\
&[E,F]=L,\quad (q-q^{-1})L=K-K^{-1}\\
&[L,E]=q(EK+K^{-1}E),\quad [L,F]=-q^{-1}(FK+K^{-1}F)
\end{aligned}
\end{align}
where $[\cdot,\cdot]$ is the commutator. We see that the algebra $U_q'(\mathfrak{sl}_2)$ is defined for all values of $q$, in particular $q=1$.
\begin{prop}\cite[Proposition VI.2.1, Proposition VI.2.1]{kassel95}
The algebra $U_q(\mathfrak{sl}_2)$ is isomorphic to the algebra $U_q'(\mathfrak{sl}_2)$. Moreover, when $q=1$, we have that
\begin{equation}U_1'(\mathfrak{sl}_2)\simeq U(\mathfrak{sl}_2)[K]/{\langle{K^2-1}\rangle}\end{equation}
\end{prop}
Therefore, in some sense, we get $U(\mathfrak{sl}_2)$ back as $q\rightarrow{1}$.
\end{exmp}

\begin{remark}
For a more general discussion of the connections between $U_q(\mathfrak{g})$ and $U(\mathfrak{g})$ see \cite[Theorem 3.4.4]{hongkang02} which considers the classical limit of $U_q(\mathfrak{g})$ as $q\rightarrow{1}$.
\end{remark}

We highlight some parallels with the classical case.
\begin{defn}\label{uqasl2}
For each $\alpha\in\Pi$ denote by $U_q^\alpha(\mathfrak{g})$ the subalgebra of $U_q(\mathfrak{g})$ generated by  $E_\alpha,F_\alpha,K_\alpha,K_\alpha^{-1}$. This subalgebra is isomorphic to $U_{q_\alpha}(\mathfrak{sl}_2)$.
\end{defn}

Let $U_q^+(\mathfrak{g})$, $U_q^0(\mathfrak{g})$, $U_q^-(\mathfrak{g})$ be the subalgebras of $U_q(\mathfrak{g})$ generated by $\{E_\alpha\}_{\alpha\in\Pi}$, $\{K_\alpha,K_\alpha^{-1}\}_{\alpha\in\Pi}$, $\{F_\alpha\}_{\alpha\in\Pi}$ respectively. As in the classical case, discussed in Proposition \ref{triangulardecomposition}, we have a triangular decomposition of $U_q(\mathfrak{g})$.

\begin{prop}\label{triangulardecomp}\cite[Theorem 3.15]{hongkang02}
$U_q(\mathfrak{g})$ has the triangular decomposition \begin{equation}U_q(\mathfrak{g})\simeq U_q^-(\mathfrak{g)}\otimes U_q^0(\mathfrak{g})\otimes U_q^+(\mathfrak{g)} \end{equation}
\end{prop}

\begin{defn}\label{borelsubquantum}
The \textbfit{positive Borel subalgebra} $U^{\geq 0}_q(\mathfrak{g})$ is the subalgebra of $U_q(\mathfrak{g})$ generated by $\{E_\alpha,K_\alpha,K_\alpha^{-1}\}_{\alpha\in\Pi}$ and the \textbfit{negative Borel subalgebra} $U^{\leq 0}_q(\mathfrak{g})$ is generated by $\{F_\alpha,K_\alpha,K_\alpha^{-1}\}_{\alpha\in\Pi}$.
\end{defn}

One of the most important properties of quantum groups is that they have a Hopf algebra structure. $U_q(\mathfrak{sl}_2)$ has the following Hopf algebra structure, where the structure maps are defined on the generators. The maps $\Delta$ and $\varepsilon$ extend to algebra homomorphisms, and $S$ extends to an algebra anti-homomorphism.
\begin{lemma}\cite[Chapter 3]{jantzen96}
$U_q(\mathfrak{sl}_2)$ has a Hopf algebra structure defined by
\begin{align}
\begin{aligned}
&\Delta(K)& &=K \otimes K,\\  
&\Delta(E)& &=E\otimes 1+K\otimes E,\quad \Delta(F)=F\otimes K^{-1} +1\otimes F,\\
&\varepsilon(K)& &=1,\quad \varepsilon(E)=\varepsilon(F)=0,\\
&S(K)& &=K^{-1},\quad S(E)=-K^{-1}E,\quad S(F)=-FK
\end{aligned}
\end{align}
\end{lemma}

By \ref{uqasl2}, there is an injective homomorphism $U_{q_\alpha}(\mathfrak{sl_2})\rightarrow{U_q(\mathfrak{g})}$. We want to give $U_q(\mathfrak{g})$ a Hopf algebra structure such that this homomorphism is a Hopf algebra homomorphism. This leads to the following construction.

\begin{prop}\cite[Proposition 4.11]{jantzen96}\label{quanthopf}
$U_q(\mathfrak{g})$, defined in \ref{quantenvalg}, has a Hopf algebra structure with $\Delta$, $\varepsilon$ and $S$ defined on $\{E_\alpha,F_\alpha,K_\alpha,K_\alpha^{-1}\}_{\alpha\in\Pi}$ by
\begin{align}
\begin{aligned}
&\Delta(K_\alpha)& &=K_\alpha \otimes K_\alpha\\
& \Delta(E_\alpha)& &=E_\alpha\otimes 1+K_\alpha\otimes E_\alpha,\quad \Delta(F_\alpha)=F_\alpha\otimes K_\alpha^{-1} +1\otimes F_\alpha,\\
&\varepsilon(K_\alpha)& &=1,\quad \varepsilon(E_\alpha)=\varepsilon(F_\alpha)=0,\\
& S(K_\alpha)& &=K_\alpha^{-1},\quad S(E_\alpha)=-K_\alpha^{-1}E_\alpha,\quad S(F_\alpha)=-F_\alpha K_\alpha
\end{aligned}
\end{align}
\end{prop}

\section{The Quantised Coordinate Ring}\label{thequantisedcoordinatering}
We now discuss an important second type of quantum group, the quantised coordinate ring of an algebraic group.

\subsection{Coordinate Rings of Algebraic Groups}

An algebraic group $G$ is defined to be an affine algebraic variety whose group structure is given by morphisms of affine varieties. Its coordinate ring $O(G)$ consists of the set of polynomial functions $f:G\rightarrow{k}$ and can be endowed with a Hopf algebra structure similar to that of $k(G)$ in Example \ref{exmp1}. 

\begin{exmp}\cite[Example I.1.5]{browngoodearl02}\label{O(sl2)} The \textbfit{special linear group of order $2$, $SL_2$}, is the algebraic group of $2\times 2$ matrices with determinant $1$, equipped with matrix multiplication. Its associated Lie algebra is $\mathfrak{sl}_2$. Define functions $x_{ij}:M_2\rightarrow{k}$ taking $2\times 2$ matrices and mapping them to their $(i,j)^{th}$ entry. In terms of these generators, the coordinate ring of $SL_2$ is \begin{equation}O(SL_2)=k[x_{11},x_{12},x_{21},x_{22}]/\langle x_{11}x_{22}-x_{12}x_{21}-1\rangle\end{equation}
$O(SL_2)$ has a Hopf algebra structure satisfying
\begin{align}
\begin{aligned}\label{osl2hopf}
& \Delta(x_{ij})& &=x_{i1}\otimes x_{1j}+x_{i2}\otimes x_{2j},\\
&\varepsilon(x_{ij})&  &=\delta_{ij},\\
& S(x_{11})& &=x_{22},\quad S(x_{12})=-x_{12},\quad S(x_{21})=-x_{21},\quad S(x_{22})=x_{11}
\end{aligned}
\end{align}
\end{exmp}

An algebraic group $G$ has an associated Lie algebra $\mathfrak{g}$. The following proposition highlights the duality between $O(G)$ and $U(\mathfrak{g})$.

\begin{prop}\cite[Theorem 3.1]{hochschild81}\label{pairingogug}
Suppose that $G$ is a semisimple, connected, simply connected algebraic group and that $\mathfrak{g}$ is its Lie algebra, then there is a perfect Hopf pairing between $O(G)$ and $U(\mathfrak{g})$. Moreover, \begin{equation}
O(G)\simeq U(\mathfrak{g})^\circ    
\end{equation}
\end{prop}

\begin{remark} \label{dualitymodcomod}
Via this duality the comodules of $O(G)$ correspond to the modules of $U(\mathfrak{g})$ by the remarks in Lemma \ref{modcomod}.
\end{remark}

However, it is not true that $O(G)^\circ\simeq U(\mathfrak{g})$. We have the weaker result that $U(\mathfrak{g})$ is isomorphic to the set of elements of $O(G)^*$ which vanish on the augmentation ideal of $O(G)$ \cite[p. 90]{browngoodearl02}.

\subsection{The Quantised Coordinate Ring}
In our following discussions, $G$ denotes a semisimple,connected, simply connected algebraic group and $\mathfrak{g}$ its corresponding Lie algebra. The motivation with our following definition of the quantised coordinate ring $O_q(G)$ is to construct a Hopf algebra yielding a Hopf pairing with $U_q(\mathfrak{g})$, mirroring the classical case discussed in Proposition \ref{pairingogug}. We make a few preliminary definitions.

\begin{defn}\cite[Definition I.7.2]{browngoodearl02}
Suppose that $V$ is a $U_q(\mathfrak{g})$-module. Given $v\in V$ and $f\in V^*$, a \textbfit{coordinate function} $c_{f,v}\in U_q(\mathfrak{g})^*$ is a linear form with \begin{equation}
c_{f,v}(u)=f(uv),\quad \text{ for all }u\in U_q(\mathfrak{g})  
\end{equation} 
\end{defn}

Let $A$ denote the set generated by the coordinate functions of a set $\mathcal{C}$ of finite dimensional $U_q(\mathfrak{g})$-modules. Suppose that $V$ is a finite-dimensional $U_q(\mathfrak{g})$-module. Its annihilator, $\text{Ann}_{U_q(\mathfrak{g})}(V)=\{u\in U_q(\mathfrak{g})\mid uv=0\text{ for all }v\in V\}$ is an ideal of finite codimension over $U_q(\mathfrak{g})$ contained in the kernel of $c_{f,v}$. Therefore, $A$ is a subspace of $U_q(\mathfrak{g})^\circ$ by Definition \ref{rd}.  Moreover, by \cite[Corollary I.7.4]{browngoodearl02} $A$ is a sub-bialgebra of $U_q(\mathfrak{g})^\circ$, and if $\mathcal{C}$ is closed under duals then $A$ is a Hopf subalgebra of $U_q(\mathfrak{g})^\circ$.\\

The finite dimensional representation theory of $U_q(\mathfrak{g})$ is very similar to that of $U(\mathfrak{g})$. We use the notation from Section \ref{quantiseduniversalenvelopingalgebra}. So, $\mathfrak{g}$ has a fixed Cartan subalgebra $\mathfrak{h}$, a root system $\Phi$, with simple roots $\Pi$.

\begin{defn}\cite[p. 42]{browngoodearl02}
\begin{enumerate}
    \item The \textbfit{weight lattice} $\Lambda$ of $\mathfrak{g}$ is $
    \Lambda=\bigoplus_{\alpha\in\Pi}\mathbb{Z}\omega_\alpha\subseteq\mathfrak{h}^*$ where $\omega_\alpha\in\mathfrak{h}^*$ are \textit{fundamental weights} satisfying $(\omega_\alpha,\beta)=\delta_{\alpha\beta}\frac{(\beta,\beta)}{2}$ for $\alpha,\beta\in\Pi$. 
    \item The set of \textbfit{dominant integral weights}, $\Lambda^+$, is the set
    $\Lambda^+=\sum_{\alpha\in\Pi}\mathbb{Z}_{\geq 0}\omega_\alpha$
\end{enumerate}
\end{defn}

\begin{defn}\cite[p. 70 ]{jantzen96}\label{type1}
Suppose that V is a finite dimensional $U_q(\mathfrak{g})$-module. We say that V is of \textbfit{type 1} if \begin{equation}V=\bigoplus_{\lambda\in\Lambda}V(\lambda), \text{ where } V(\lambda)=\{v\in V\mid K_{\mu}v=q^{(\lambda,\mu)}v \text{ for all } \mu\in \mathbb{Z}\Phi\}.\end{equation}
We say that $V(\lambda)$ is a \textit{weight space} of $V$ with \textit{weight} $\lambda$. Given $v\in V$, if $E_\alpha v=0$ for all $\alpha\in\Pi$ then $v$ is a \textit{highest weight vector} of \textit{highest weight} $\lambda$ and the module generated by $v$ is a \textit{highest weight module}. 
\end{defn}
The following theorem shows that the finite dimensional simple type 1 $U_q(\mathfrak{g})$-modules are indexed by dominant weights and are highest weight modules.
\begin{thm}\cite[p. 56]{browngoodearl02}
There is a bijection $\lambda\rightarrow{V(\lambda)}$ between the set of dominant integral weights $\Lambda^+$ and finite dimensional simple type 1 $U_q(\mathfrak{g})$-modules. Further, $V(\lambda)$ has a highest weight vector $u_\lambda$ of weight $\lambda$.
\end{thm}
We are now ready to define $O_q(G)$.
\begin{defn}\label{quantisedcoordring}
The \textbfit{quantised coordinate ring} $O_q(G)$ is the subalgebra of $U_q(\mathfrak{g})^{\circ}$ generated by the coordinate functions of  $V(\lambda)$ for $\lambda\in\Lambda^+$.
\end{defn}

We note that there is a Hopf pairing.
\begin{equation}
(\cdot,\cdot):U_q(\mathfrak{g})\times O_q(G)\rightarrow{k}, \text{ given by }(u,c_{f,v})=f(uv)    
\end{equation}
for $u\in U_q(\mathfrak{g}),c_{f,v}\in O_q(G)$. Further, as the $V(\lambda)$ are closed under duals by \cite[p. 57]{browngoodearl02}, we see that $O_q(G)$ is a Hopf subalgebra of $U_q(\mathfrak{g})^\circ$. Therefore, $O_q(G)$ inherits a Hopf algebra structure from that of $U_q(\mathfrak{g})^\circ$, as defined in Section \ref{hopfdual}. \\

The following theorem is a quantum analogue of the classical Peter-Weyl theorem. It allows us to examine the crystal basis theory of $O_q(G)$, discussed in Section \ref{thecategoryofcrystals}.

\begin{thm}\cite[Proposition 4.1]{ganev16}\label{qpw}
(Quantum Peter-Weyl Theorem) There is an isomorphism of $U_q(\mathfrak{g})$-$U_q(\mathfrak{g})$-bimodules \begin{equation}\phi:\bigoplus_{\lambda\in\Lambda^+}V(\lambda)\otimes V(\lambda)^*\rightarrow{O_q(G)}\end{equation}
\end{thm}

\begin{remark}
We see that the right $U_q(\mathfrak{g})$-module action on $O_q(G)$ comes from the action of $U_q(\mathfrak{g})$ on $V(\lambda)^*$, while the left action comes from the action of $U_q(\mathfrak{g})$ on $V(\lambda)$.
\end{remark}

\subsection{$O_q(SL_2)$: The Quantised Coordinate Ring of $SL_2$}

We would like to use Definition \ref{quantisedcoordring} to construct the quantised coordinate ring $O_q(SL_2)$, where $SL_2=SL_2(k)$. We first construct $O_q(SL_2(\mathbb{C}))$. The fundamental weight of $\mathfrak{sl}_2(\mathbb{C})$ is $\omega=\frac{\alpha}{2}$, where $\alpha$ is defined in Example \ref{usl2}. We identify the dominant weight lattice $\Lambda^+=\mathbb{Z}_{\geq 0}\omega$ with $\mathbb{Z}_{\geq 0}$ and state the following propositions regarding simple $U_q(\mathfrak{sl}_2)$-modules.

\begin{prop}\cite[Section 4.2.6]{hongkang02}\cite[Section 1.5]{kashiwara95}\label{irreducibles}
For each $n\in\mathbb{Z}_{\geq 0}$, there is an \\$(n+1)$-dimensional simple $U_q(\mathfrak{sl}_2)$-module $V(n)$ with basis $\{u,Fu,\dots,F^{(n)}u\}$ satisfying the relations
\begin{equation}
    Eu=0,\quad Ku=q^nu,\quad F^{(k)}u=\frac{1}{[k]_q!}F^ku
\end{equation}
where $[k]_q!$ is defined in Definition \ref{quantenvalg}.
\end{prop}
\begin{thm}\label{generate}\cite[Theorem 1.1]{kashiwara95} Every simple type 1 $(n+1)$-dimensional $U_q(\mathfrak{sl}_2)$-module is isomorphic to $V(n)$ for some $n\in\mathbb{Z}_{\geq 0}$.
\end{thm}

We also have the following important proposition regarding modules over general $U_q(\mathfrak{g})$.
\begin{prop}\label{finitgen}\cite[Proposition I.7.8]{browngoodearl02}
If $\lambda_1,\dots,\lambda_k$ generate $\Lambda^+$, then the coordinate functions of $V(\lambda_1),\dots,V(\lambda_k)$ generate $O_q(G)$.
\end{prop}

Therefore, we see that $O_q(SL_2(\mathbb{C}))$ is the $k$-algebra generated by the coordinate functions of $V(1)$. If we relabel the basis as $\{v_1,v_2\}$ then our notation corresponds to that used in $\cite[Section 1.7]{browngoodearl02}$, where $O_q(SL_2(\mathbb{C}))$ is constructed using the coordinate functions $x_{ij}=c_{f_i,v_j}$, with $\{f_1,f_2\}$ the dual basis. Moreover, \cite[Theorem I.7.16]{browngoodearl02} shows that $O_q(SL_2(\mathbb{C}))$ is isomorphic to $O_q(SL_2(k))$, and so we can unambiguously define $O_q(SL_2)$ as follows.

\begin{defn}\cite[Section I.1.6]{browngoodearl02}\label{oq(sl2)}
$O_q(SL_2)$ is generated by $\{X_{11},X_{12},X_{21},X_{22}\}$ with the relations
\begin{align}
\begin{aligned}
& X_{11}X_{12}=qX_{12}X_{11}, \qquad X_{11}X_{21}=qX_{21}X_{11},\\
&X_{12}X_{22}=qX_{22}X_{12}, \qquad X_{21}X_{22}=qX_{22}X_{21},\\
&X_{12}X_{21}=X_{21}X_{12},\\
&X_{11}X_{22}-X_{22}X_{11}=(q-q^{-1})X_{12}X_{21},\\
&X_{11}X_{22}-qX_{12}X_{21}=1
\end{aligned}
\end{align}
\end{defn}

\begin{remark}
When $q=1$, we obtain our relations for the classical coordinate ring $O(SL_2)$ from Example \ref{O(sl2)}. We see that the first six express commutativity, whereas the last is the condition $X_{11}X_{22}-X_{12}X_{22}=1$.
\end{remark}

\begin{prop}\cite[Section I.7.10]{browngoodearl02}\label{hopfoqsl2}
$O_q(SL_2)$ has a Hopf algebra structure satisfying
\begin{align}
\begin{aligned}
&\Delta(X_{ij})& &=X_{i1}\otimes X_{1j}+X_{i2}\otimes X_{2j},\\
&\varepsilon(X_{ij})& &=\delta_{ij},\\
&S(X_{11})& &=X_{22},\quad S(X_{12})=-q^{-1}X_{12},\quad S(X_{21})=-qX_{21},\quad S(X_{22})=X_{11}
\end{aligned}
\end{align}
\end{prop}

The duality between $O_q(SL_2)$ and $U_q(\mathfrak{sl}_2)$ is illustrated by the following proposition.

\begin{prop}\cite[Chapter VII.4]{kassel95}\label{pairingsl2}
There is a perfect Hopf pairing between $U_q(\mathfrak{sl}_2)$ and $O_q(SL_2)$.
\end{prop}
\begin{proof}
The bilinear form  $(\cdot\,,\cdot):U_q(\mathfrak{sl}_2)\times O_q(SL_2)\rightarrow{k}$ acts on the generators of $U_q(\mathfrak{sl}_2)$ and $O_q(SL_2)$ as follows
\begin{align}
\begin{aligned}
(E,X_{11})&=0,  &  (E,X_{12})&=1,      &  (E,X_{21})&=0,   &  (E,X_{22})&=0\\
(F,X_{11})&=0,  &  (F,X_{12})&=0,      &  (F,X_{21})&=1,   &  (F,X_{22})&=0\\
(K,X_{11})&=q,  &  (K,X_{12})&=0,      &  (K,X_{21})&=0,   &  (K,X_{22})&=q^{-1}\\
(K^{-1},X_{11})&=q^{-1},  &  (K^{-1},X_{12})&=0, &  (K^{-1},X_{21})&=0,   &  (K^{-1},X_{22})&=q
\end{aligned}
\end{align}
This can be extended linearly to give the Hopf pairing between $U_q(\mathfrak{sl}_2)$ and $O_q(SL_2)$. We see that, as $q$ is not a root of unity, this pairing is perfect. For each generator of $U_q(\mathfrak{sl}_2)$, we can find a generator of $O_q(SL_2)$ such that their pairing is non-zero, and vice versa. See the cited text for more details.
\end{proof}

\section{Subgroups of Quantum Groups}\label{subgroupsofquantumgroups}
In this section, we consider a semisimple, connected, simply connected algebraic group $G$ with corresponding Lie algebra $\mathfrak{g}$. Their associated quantum groups are $O_q(G)$ and $U_q(\mathfrak{g})$ respectively.\\

The definition of a subgroup of a quantum group is quite hazy in the literature.  M\"uller in \cite{muller00} and Parshall and Wang in \cite{parshallwang91}, to name a few, consider subgroups of quantum groups as Hopf algebras satisfying some extra conditions. In this chapter, we discuss a definition of a subgroup of a quantum group discussed by Christodoulou in \cite{christodoulou16}. Her work particularly expands on work done by Letzter  \cite{letzter02} and M\"uller and Schneider \cite{mullerschneider99} on right coideal subalgebras.

\subsection{Coideal Subalgebras and Quotient Coalgebras}\label{coidealquotient}

We make the following definitions for a Hopf algebra $H=(H,\mu,\eta,\Delta,\varepsilon, S)$. We note that left/right can be interchanged consistently in our propositions and definitions in the following sections.

\begin{defn}\cite[Definition 2.1]{christodoulou16}
 A subspace $I$ of $H$ is a \textbfit{right coideal} if $\Delta(I)\subseteq I\otimes H$. A subspace $I$ of $H$ is a \textbfit{coideal} if $\Delta(I)\subseteq I\otimes H+H\otimes I\text{ and } \varepsilon(I)=0$.
\end{defn}

\begin{defn}
A subalgebra $A$ of $H$ is a \textbfit{right coideal subalgebra} of $H$ if $A$ is a right coideal of $H$.
\end{defn}
\begin{prop}\label{subalgebracoideal}\cite[Theorem 4.1]{letzter02}
If $\mathfrak{g}$ is a Lie algebra then the set of one-sided coideal subalgebras of $U(\mathfrak{g})$ is the set of enveloping algebras $U(\mathfrak{g'})$ for Lie subalgebras $\mathfrak{g'}$ of $\mathfrak{g}$.
\end{prop}

\begin{defn}\cite[Definition 2.3]{christodoulou16}
$C$ is a \textbfit{quotient left $H$-module coalgebra} if $C$ is the quotient of $H$ by a coideal and left ideal $I$.
\end{defn}

\begin{remark}
The coalgebra structure is given by comultiplication, $\bar{\Delta}(C)\rightarrow{C\otimes C}$ with $\bar{\Delta}(h+I)=\Delta(h)+I\otimes H+H\otimes I$ for $h\in H$, and counit $\bar{\varepsilon}:C\rightarrow{k}$, with $\bar{\varepsilon}(h+I)=\varepsilon(h)$ for $h\in H$. 
\end{remark}

We state the following proposition showing the one-to-one correspondence between right coideal subalgebras of $H$ and quotient left $H$-module coalgebras.

\begin{prop}\cite[Proposition 1]{takeuchi79}\label{prop2}
\begin{enumerate}
    \item Suppose that $A$ is a right coideal subalgebra of $H$. Let $A^+=A\cap {\normalfont{\textrm{Ker }}}\varepsilon$. Then $C=H/HA^+$ is a quotient left module coalgebra of $H$.
    \item If $C$ is a quotient left $H$-module coalgebra, denote by $\pi:H\rightarrow{C}$ the quotient map. Then, $A={}^CH=\{x\in H\mid \sum\pi(x_{(1)})\otimes x_{(2)}=\pi(1)\otimes x\}$ is a right coideal subalgebra of $H$.
\end{enumerate}
\end{prop}

Recall that there is a Hopf pairing $(\cdot,\cdot):U_q(\mathfrak{g})\times O_q(G)\rightarrow{k}$. For $a\in O_q(G)$, $\Delta(a)=\sum a_{(1)}\otimes a_{(2)}$ and so we can define a $U_q(\mathfrak{g})$-$U_q(\mathfrak{g})$-bimodule structure on $O_q(G)$ in the following way
\begin{equation}a\cdot u=\sum (u,a_{(1)})a_{(2)},\quad u\cdot a=\sum a_{(1)}(u,a_{(2)})\end{equation}
for $a\in O_q(G)$ and $u\in U_q(\mathfrak{g})$. Since multiplication in $O_q(G)$ is a convolution we have that
\begin{equation}(ab)\cdot u=\sum (a\cdot u_{(1)})(b\cdot u_{(2)})\quad u\cdot (ab)=\sum (u_{(1)}\cdot a)(u_{(2)}\cdot b)\end{equation}
for $a,b\in O_q(G)$ and $u\in U_q(\mathfrak{g})$.

\begin{prop}\cite[Theorem 3.1]{letzter02}\label{prop1}
For any right coideal $I$ of $U_q(\mathfrak{g})$ the subset $A=\{a\in O_q(G)\mid a\cdot u=\varepsilon(u)a \text{ for all } u\in I\} $ is a right coideal subalgebra of $O_q(G)$.
\end{prop}
\begin{proof}
We first show that $A$ is a subalgebra. Suppose that $a,b\in A$ and $u\in I$. We want to show that $ab\in A$.
\begin{align}
(ab)\cdot u &=\sum (a \cdot u_{(1)})(b\cdot u_{(2)})\nonumber\\
\intertext{Since $I$ is a right coideal of $U_q(\mathfrak{g})$, $u_{(1)}\in I$. So, since $a\in A$,}\nonumber
            &=\sum (\varepsilon(u_{(1)}) a)(b\cdot u_{(2)})\nonumber\\
            &= \sum a(b\cdot\varepsilon(u_{(1)})u_{(2)})
\intertext{Since $\sum \varepsilon(u_{(1)})u_{(2)}=u$,}               &= a(b\cdot u)\nonumber\\
            &= a\varepsilon(u)b\nonumber\\
            &= \varepsilon(u)(ab)  \nonumber
\end{align}

So $ab\in A$. Also, $1_{O_q(G)}\in A$ since $1_{O_q(G)}$ is the coordinate function of the trivial representation and hence, $1_{O_q(\mathfrak{g})}\cdot u=\varepsilon(u)$ for all $u\in I$. Therefore, $A$ is a subalgebra. \\

We now check that $A$ is a right coideal. Suppose that $\Delta(a)=\sum a_{(1)}\otimes a_{(2)}$ for $a\in A$. We want to show that $a_{(1)}\in A$. For $u\in I$ we have 
\begin{align}
\Delta(a\cdot u)&=\Delta(a)(u\otimes 1_{U_q(\mathfrak{g})})=\sum a_{(1)}\cdot u\otimes a_{(2)} \\
\intertext{and,}
\Delta(a\cdot u)&=\Delta(\varepsilon(u)a)=\sum \varepsilon(u)a_{(1)}\otimes a_{(2)}.
\end{align}

Therefore, if we assume that the $a_{(2)}$ are linearly independent
\begin{equation}
    a_{(1)}\cdot u=\varepsilon(u)a_{(1)}
\end{equation} So, $a_{(1)}\in A$. Hence $\Delta(A)\subseteq A\otimes O_q(G)$.
\end{proof}

\begin{remark}
The subalgebra $A$ is the set of right invariants of $I$ on $O_q(G)$ and is often denoted $O_q(G)^I$.
\end{remark}

We see from Propositions \ref{prop1} and \ref{prop2} that for any right coideal $I$ in $U_q(\mathfrak{g})$, we can define a right coideal subalgebra of $O_q(G)$, equivalently a quotient left $O_q(G)$-module coalgebra.

\subsection{Definition of a Subgroup of a Quantum Group}\label{subgroups}

Suppose that we have a subgroup $G'$ of $G$ with corresponding Lie subalgebra $\mathfrak{g'}$ of $\mathfrak{g}$. We want to find their corresponding quantum groups $O_q(G')$ and $U_q(\mathfrak{g'})$ respectively. The enveloping algebra $U(\mathfrak{g}')$ is a subalgebra of $U(\mathfrak{g})$ by Lemma \ref{subalgebracoideal}, moreover it is a Hopf subalgebra. Therefore, mirroring the classical case, it seems intuitive to first look at corresponding Hopf subalgebras of $U_q(\mathfrak{g})$ or, dually, quotient Hopf algebras of $O_q(G)$.

\begin{exmp}\label{borelllll}
Suppose that we have a Borel subgroup $G^{\geq 0}$ of $G$ with Borel Lie subalgebra $\mathfrak{g}^{\geq 0}=\mathfrak{h}\oplus\sum_{\alpha\in\Phi^+}\mathfrak{g}_\alpha$ of $\mathfrak{g}$, see Equation \ref{cartandecomp}. The corresponding universal enveloping algebra is $U^{\geq 0}(\mathfrak{g})$, with quantum analogue $U^{\geq 0}_q(\mathfrak{g})$ as defined in Section \ref{quantiseduniversalenvelopingalgebra}. $U^{\geq 0}_q(\mathfrak{g})$ is a Hopf subalgebra of $U_q(\mathfrak{g})$ since it is closed under $\Delta,S$. The quantised coordinate ring $O_q^{\geq 0}(G)$ can be constructed using $U_q^{\geq 0}(\mathfrak{g})$-modules as in Section \ref{thequantisedcoordinatering} and can be considered as the quantum subgroup corresponding to $G^{\geq 0}$.
\end{exmp}

In the above example, $U_q^{\geq 0}(\mathfrak{g})$ is a Hopf subalgebra of $U_q(\mathfrak{g})$. However, generally, $U_q(\mathfrak{g'})$ is not isomorphic to a Hopf subalgebra of $U_q(\mathfrak{g})$. Moreover, there are many subalgebras of $U_q(\mathfrak{g})$ which are not Hopf subalgebras but can  still be considered as quantisations of $U(\mathfrak{g'})$.

\begin{exmp}\label{example}
Suppose that we have a nilpotent subgroup $G^-$ of $G$ with nilpotent Lie subalgebra $\mathfrak{g}^-=\sum_{\alpha\in\Phi^-}\mathfrak{g}_\alpha$ of $\mathfrak{g}$. The corresponding universal enveloping algebra is $U^-(\mathfrak{g})$, with quantum analogue $U^-_q(\mathfrak{g})$ as defined in Section \ref{quantiseduniversalenvelopingalgebra}. However, $U^-_q(\mathfrak{g})$ is not a Hopf subalgebra of $U_q(\mathfrak{g})$ since $S(F_\alpha)=-F_\alpha K_\alpha\not\in U_q^-(\mathfrak{g})$ for each $\alpha\in\Pi$. 
\end{exmp}
\begin{remark}
We remark that $U_q^-(\mathfrak{g})$ is a right coideal subalgebra of $U_q(\mathfrak{g})$.
\end{remark}

The above example shows that defining subgroups of $U_q(\mathfrak{g})$ to be Hopf subalgebras is too restrictive and therefore we need to broaden our definition.

By Lemma \ref{subalgebracoideal}, $U(\mathfrak{g'})$ is a one-sided coideal subalgebra of $U(\mathfrak{g})$. Suppose, without loss of generality, that it is a left coideal subalgebra. We can often find a right coideal subalgebra of $U_q(\mathfrak{g})$ corresponding to $U(\mathfrak{g'})$ \cite[p. 9]{letzter02}. This gives rise to a right coideal subalgebra of $O_q(G)$ by Proposition \ref{prop1}, and by Proposition \ref{prop2} a corresponding quotient left $O_q(G)$-module coalgebra. This informs the following definition.
\begin{defn}
A \textbfit{subgroup of $O_q(G)$} is a quotient left $O_q(G)$-module coalgebra. 
\end{defn}

As $O_q(G)$ lies in the dual of $U_q(\mathfrak{g})$, we would like our definitions of subgroups to line up with the general ideal that `sub-objects' induce `quotient objects' in the dual. We have changed the definition of a subgroup of $U_q(\mathfrak{g})$ given in \cite{christodoulou16} to emphasise this duality.

\begin{defn}
A \textbfit{subgroup of $U_q(\mathfrak{g})$} is a right coideal subalgebra of $U_q(\mathfrak{g})$.
\end{defn}

\subsection{Subgroups of $U_q(\mathfrak{sl}_2)$}\label{subgroupsuq(sl2)}

We want to find subgroups of $U_q(\mathfrak{sl}_2)$.  In \cite{vocke18}, Vocke discusses right coideal subalgebras, equivalently subgroups, of $U_q(\mathfrak{g})$, and in particular finds all the right coideal subalgebras of $U_q(\mathfrak{sl}_2)$. Her work builds upon work done in \cite{heckenbergerkolb11} on right coideal subalgebras of $U_q(\mathfrak{g})$ containing $U_q^0(\mathfrak{g})$ \\

Vocke constructs a set of possible generating elements of right coideal subalgebras of $U_q(\mathfrak{sl}_2)$
\begin{equation}
\{K^j,EK^{-1},EK^{-1}+\lambda K^{-1},EK^{-1}+c_FF+c_KK^{-1}\}    
\end{equation}
up to symmetry in $E$ and $F$ with $j\in\mathbb{Z}$ and $\lambda,c_F,c_K\in k$.\\

The complete list of right coideal subalgebras of $U_q(\mathfrak{sl}_2)$ is
\begin{equation}
U_q(\mathfrak{sl}_2), \quad U_q^0(\mathfrak{sl}_2),\quad U^{\geq 0}_q(\mathfrak{sl}_2),\quad U^{\leq 0}_q(\mathfrak{sl}_2),
\end{equation}
And, the subalgebras generated by the following sets
\begin{equation}
\{EK^{-1},K^2,K^{-2},F\}   
\end{equation}
For $j\in\mathbb{Z}$
\begin{equation}
\{ EK^{-1},K^j\}, \qquad \{ F,K^j\},  
\end{equation}
For $\lambda,\lambda'\in k$
\begin{equation}
\{ EK^{-1}+\lambda K^{-1}\}, \qquad  \{ F+\lambda'K^{-1}\}
\end{equation}
For $c_F,c_K\in k$
\begin{equation}
    \{ EK^{-1}+c_FF+c_KK^{-1}\}
\end{equation}
And, for $\lambda,\lambda'\in k$ with $\lambda\lambda'=\frac{q^2}{(1-q^2)(q-q^{-1})}$ 
\begin{equation}
\{ EK^{-1}+\lambda K^{-1},F+\lambda'K^{-1}\}    
\end{equation}

\subsection{Subgroups of $O_q(SL_2)$}

We now want to find subgroups of $O_q(SL_2)$ corresponding to the right coideal subalgebras of $U_q(\mathfrak{sl}_2)$. One method is to find the set of invariants of the action of each right coideal $I$ of $U_q(\mathfrak{sl}_2)$ on $O_q(SL_2)$. By Proposition \ref{prop2}, this is a right coideal subalgebra of $O_q(SL_2)$, which gives rise to a quotient left $O_q(SL_2)$-module coalgebra.\\\

Consider the $U_q(\mathfrak{sl_2})$-$U_q(\mathfrak{sl}_2)$-bimodule action on $O_q(SL_2)$ discussed in Section \ref{coidealquotient}. Since $\Delta(X_{ij})=X_{i1}\otimes X_{1j}+X_{i2}\otimes X_{2j}$,
\begin{equation}X_{ij}\cdot u=(u,X_{i1})X_{1j}+(u,X_{i2})X_{2j}\end{equation}
and, using Proposition \ref{pairingsl2}, 
\begin{align}
\begin{aligned}
X_{1j}\cdot E&=X_{2j}, & X_{1j}\cdot F&=0,& X_{1j}\cdot K&=qX_{1j}, &X_{1j}\cdot K^{-1}&=q^{-1}X_{1j}\\
X_{2j}\cdot E&=0, & X_{2j}\cdot F&=X_{1j},& X_{2j}\cdot K&=q^{-1}X_{2j}, &X_{2j}\cdot K^{-1}&=qX_{2j}
\end{aligned}
\end{align}
for $j=1,2$. And, for $u,v\in U_q(\mathfrak{sl}_2)$, we have
\begin{align}
X_{ij}\cdot (uv)&=(uv,X_{i1})X_{1j}+(uv,X_{i2})X_{2j}
\intertext{with, by definition of the Hopf pairing,}
(uv,X_{ij})&=(u,X_{i1})(v,X_{1j})+(u,X_{i2})(v,X_{2j})
\end{align}

\begin{exmp}
Consider the case when $I$ is the right coideal subalgebra generated by $EK^{-1}$. Since $\varepsilon(EK^{-1})=0$, we want to find \begin{equation}
A=\{a\in O_q(SL_2)\mid a\cdot u=0 \text{ for all } u\in I\}    
\end{equation} 
As, \begin{equation}(EK^{-1},X_{11})=0,\,(EK^{-1},X_{12})=q,\,(EK^{-1},X_{21})=0,\, (EK^{-1},X_{22})=0
\end{equation}
we have
\begin{equation}X_{11}\cdot (EK^{-1})=qX_{21},\quad X_{12}\cdot (EK^{-1})=qX_{22},\quad X_{21}\cdot (EK^{-1})=0,\quad X_{22}\cdot (EK^{-1})=0\end{equation}
From \cite[Example I.11.8]{browngoodearl02}, we see that:
\begin{equation}
    \{X_{22}^iX_{11}^jX_{21}^k\mid i,j,k\in\mathbb{Z}_{\geq 0}\}\cup \{X_{22}^iX_{22}^lX_{21}^k\mid i,k,l\in\mathbb{Z}_{\geq 0}, l>0\}
\end{equation}
is a basis for $O_q(SL_2)$. Since $X_{21},X_{22}\in A$ but $X_{11},X_{12}\not\in A$, the only basis elements contained in $A$ are those containing $X_{22}$ or $X_{21}$. Therefore, $A$ is the ideal of $O_q(SL_2)$ generated by $X_{21},X_{22}$. This is a right coideal subalgebra of $O_q(SL_2)$ by Proposition \ref{prop1} and further, since $A\cap$ {\normalfont{Ker}} $\varepsilon =A$, 
\begin{equation}
O_q(SL_2)/\langle X_{21},X_{22}\rangle    
\end{equation}
is a subgroup of $O_q(SL_2)$ by Proposition \ref{prop2}.

\end{exmp}

We have seen that $U_q^\pm(\mathfrak{sl}_2)$, $U_q^0(\mathfrak{sl}_2)$, $U_q^{\geq 0}(\mathfrak{sl_2})$, $U_q^{\leq 0}(\mathfrak{sl_2})$ are subgroups of $U_q(\mathfrak{sl_2})$. We now discuss the corresponding subgroups of $O_q(SL_2)$.\\\

The Borel subgroups $SL_2^{\geq 0}$, $SL_2^{\leq 0}$ consist of upper, respectively lower, triangular matrices. The unipotent radicals $SL_2^+$, $SL_2^+$ in $SL_2^{\geq 0}$, $SL_2^{\leq 0}$ respectively consist of unipotent matrices.
The coordinate rings of these subgroups are
\begin{align}
\begin{aligned}
O^{\geq 0}(SL_2)&\simeq O(SL_2)/\langle X_{21}\rangle, &  O^+(SL_2)&\simeq O^{\geq 0}(SL_2)/\langle X_{11}-1,X_{22}-1\rangle,
\\O^{\leq 0}(SL_2)&\simeq O(SL_2)/\langle X_{12}\rangle,
 &  O^-(SL_2)&\simeq O^{\leq 0}(SL_2)/\langle X_{11}-1,X_{22}-1\rangle
\end{aligned}
\end{align}
The standard maximal torus $SL_2^0$ is the subgroup of $SL_2$ consisting of diagonal matrices. Its coordinate ring is defined by
\begin{equation}O^0(SL_2)=O(SL_2)/\langle X_{12},X_{21}\rangle\end{equation}
In analogy with the classical case, we make the following definitions \cite[p. 227- 228]{jaramillo14}
\begin{equation}
\begin{aligned}
O^{\geq 0}_q(SL_2)&=O_q(SL_2)/{\langle X_{21} \rangle},& \quad O_q^+(SL_2)&=O_q^{\geq 0}(SL_2)/\langle X_{11}-1,X_{22}-1\rangle\\
O^{\leq 0}_q(SL_2)&=O_q(SL_2)/{\langle X_{12} \rangle},&\quad O_q^-(SL_2)&= O_q^{\leq 0}(SL_2)/\langle X_{11}-1,X_{22}-1\rangle\\
\end{aligned}
\end{equation}
And, 
\begin{equation}
O_q^0(SL_2)=O_q(SL_2)/\langle X_{12},X_{21}\rangle    
\end{equation}
where we use the notation $\langle\quad\rangle$ to denote the left ideal generated by a set. It is an easy check to show that the ideals are coideals and therefore the above equations define quotient left $O_q(SL_2)$-module coalgebras, equivalently subgroups of $O_q(SL_2)$. We note that as $q\rightarrow{1}$, we obtain our classical coordinate rings. \\

\section{A Categorical Approach to Subgroups of Quantum Groups}\label{categoricalapproachtosubgroupsofquantumgroups}
The aim of this section is to provide a categorical characterisation of a subgroup of a quantum group. The basic definitions can be found in \cite{safronov18,maclane71}.

\subsection{Monoidal Categories}

Here we discuss a particular type of category known as a monoidal category. It can be considered as a `categorification' of the notion of a monoid, a set with an associative multiplication and an identity element. We loosely interpret `categorification' as the process of turning sets into categories by adding morphisms.

\begin{defn}\cite[Definition 2.2.8]{etingof15}
A \textbfit{monoidal category} $(\mathcal{C}, \otimes,I, \alpha, l, r)$ consists of:
\begin{itemize}
    \item A category $\mathcal{C}$,
    \item A bifunctor $\otimes:\mathcal{C}\times \mathcal{C}\rightarrow{\mathcal{C}}$,
    \item A unit object $I\in \mathcal{C}$,
    \item A natural isomorphism $\alpha$ with morphism $\alpha_{X,Y,Z}:(X\otimes Y)\otimes Z\rightarrow{X\otimes (Y\otimes Z)}$ for $X,Y,Z \in$ Ob $\mathcal{C}$,
    \item Natural isomorphisms $l$ and $r$ such that, for each $X\in$ Ob $\mathcal{C}$, we have morphisms: $l_X:I\otimes X\rightarrow{X}$ and $r_X:X\otimes I\rightarrow{X}$.
\end{itemize}
Such that the following diagrams
\begin{center}
{\begin{tikzpicture}
    [
      border rotated/.style = {shape border rotate=180}
    ]
    \node[
      regular polygon,
      regular polygon sides=5,
      border rotated,
      minimum width=70mm,
    ] (PG) {}
      (PG.corner 1) node (PG9) {$X\otimes((Y\otimes Z)\otimes W$)}
      (PG.corner 2) node (PG3) {$X\otimes (Y\otimes (Z\otimes W$))}
      (PG.corner 3) node (PG2) {$(X\otimes Y)\otimes (Z\otimes W)$}
      (PG.corner 4) node (PG1) {$((X\otimes Y)\otimes Z)\otimes W$}
      (PG.corner 5) node (PG0) {$(X\otimes (Y\otimes Z))\otimes W$}
    ;
     \draw[transform canvas={yshift=0.5ex},->] (PG1) -- (PG2) node [midway,above] {\footnotesize $\alpha_{X\otimes Y,Z,W}$};
     \draw[->] (PG0) -- (PG9) node [midway,sloped,below] {\footnotesize $\alpha_{X,Y\otimes Z,W}$};
     \draw[->] (PG1) -- (PG0) node [midway,sloped,above] {\footnotesize $\alpha_{X,Y,Z}\otimes \text{id}_W$};
     \draw[->] (PG2) -- (PG3) node [midway,sloped,above] {\footnotesize $\alpha_{X,Y,Z\otimes W}$};
     \draw[->] (PG9) -- (PG3) node [midway,sloped,below] {\footnotesize $\text{id}_X \otimes \alpha_{Y,Z,W}$};
\end{tikzpicture}}
The Pentagonal Identity
\end{center}
and
\begin{center}
\begin{tikzcd}
    (X\otimes I)\otimes Y \arrow{rr}{ \alpha_{X,I,Y}} \arrow[swap]{dr}{ r_X\otimes \text{id}_Y} & & X\otimes(I\otimes Y) \arrow{dl}{ \text{id}_X\otimes l_Y } \\
    & X\otimes Y
\end{tikzcd}

The Triangular Identity
\end{center}
commute.
\end{defn}

\begin{exmp}\label{repg} If $G$ is a group, then {\normalfont{\textbf{Rep}$(G)$}} is the category with representations of $G$ as its objects and \textit{equivariant maps} as morphisms. An equivariant map between two representations $(V,\rho)$ and $(W,\mu)$ is a linear map $f:V\rightarrow{W}$ satisfying,
    \begin{equation}f\circ\rho(g)=\mu(g)\circ f\end{equation}
    for all $g\in G$.
The category {\normalfont{\textbf{Rep$(G$)}}} can be given the structure of a monoidal category. We define $\otimes$ to be the tensor product of representations. For $(V,\rho)$ and $(W,\mu)$ we have
    \begin{equation}\rho\otimes\mu:G\rightarrow{GL(V\otimes W)},\quad(\rho\otimes\mu)(g)=\rho(g)\otimes\mu(g)\end{equation}
    The unit object, $I$, is the trivial representation $I=(k,\rho)$ where $\rho(g)v=v$ for all $g\in G, v\in k$. The associativity and identity conditions follow from properties of the tensor product.
\end{exmp}

\subsection{Module Categories}

We now consider module categories, which can be interpreted as the `categorification' of the notion of a module over a monoid. 

\begin{defn}\cite[Definition 7.1.2]{etingof15}
A \textbfit{(left) module category} $(\mathcal{M}, \otimes, \mu, l)$ over a monoidal category $\mathcal{C}$ consists of:
\begin{itemize}
    \item A category $\mathcal{M}$,
    \item A bifunctor $\otimes:\mathcal{C\times M}\rightarrow{\mathcal{M}}$,
    \item A natural isomorphism $\mu$ with morphism $\mu_{X,Y,M}:(X\otimes Y)\otimes M\rightarrow{X\otimes(Y\otimes M)}$ for $X,Y\in\mathcal{C}$ and $M\in\mathcal{M}$,
    \item A natural isomorphism $l$ such that, for each $M\in$ Ob $\mathcal{M}$, we have a morphism $l_M:I\otimes M\rightarrow{M}$.
\end{itemize}    
Such that the following diagrams
\begin{center}
\begin{tikzpicture}
    [
      border rotated/.style = {shape border rotate=180}
    ]
    \node[
      regular polygon,
      regular polygon sides=5,
      border rotated,
      minimum width=70mm,
    ] (PG) {}
      (PG.corner 1) node (PG9) {$X\otimes((Y\otimes Z)\otimes M$)}
      (PG.corner 2) node (PG3) {$X\otimes (Y\otimes (Z\otimes M$))}
      (PG.corner 3) node (PG2) {$(X\otimes Y)\otimes (Z\otimes M)$}
      (PG.corner 4) node (PG1) {$((X\otimes Y)\otimes Z)\otimes M$}
      (PG.corner 5) node (PG0) {$(X\otimes (Y\otimes Z))\otimes M$}
    ;
     \draw[transform canvas={yshift=0.5ex},->] (PG1) -- (PG2) node [midway,above] {\footnotesize $\mu_{X\otimes Y,Z,M}$};
     \draw[->] (PG0) -- (PG9) node [midway,sloped,below] {\footnotesize $\mu_{X,Y\otimes Z,M}$};
     \draw[->] (PG1) -- (PG0) node [midway,sloped,above] {\footnotesize $\mu_{X,Y,Z}\otimes \text{id}_M$};
     \draw[->] (PG2) -- (PG3) node [midway,sloped,above] {\footnotesize $\mu_{X,Y,Z\otimes M}$};
     \draw[->] (PG9) -- (PG3) node [midway,sloped,below] {\footnotesize $\text{id}_X \otimes \mu_{Y,Z,M}$};
\end{tikzpicture}
\end{center}
and
\medskip
\begin{center}
\begin{tikzcd}
    (X\otimes I)\otimes M \arrow{rr}{\mu_{X,I,M}} \arrow[swap]{dr}{r_X\otimes \text{id}_M} & & X\otimes(I\otimes M) \arrow{dl}{\text{id}_X\otimes l_M } \\
    & X\otimes M
\end{tikzcd}

\end{center}
commute.
\end{defn}

\begin{prop}\label{rephrepg}
If $G$ is a finite group and $H$ is a subgroup of $G$, then \normalfont{\textbf{Rep}$(H)$} \textit{is a module category over} \normalfont{\textbf{Rep}$(G)$}.
\end{prop}
\begin{proof}
Define the restriction functor
\begin{equation}
\text{Res}^G_H:\textbf{Rep$(G)$}\rightarrow{\textbf{Rep$(H)$}},\quad \text{Res}^G_H(\rho)=\rho\restrict{H}   
\end{equation}We define our bifunctor 
$\otimes_{\textbf{Rep}(H)}:\textbf{Rep}(G)\times \textbf{Rep}(H)\rightarrow{\textbf{Rep}(H)}$ by \begin{equation}
\otimes_{\textbf{Rep}(H)}(\rho\times \mu)=\text{Res}^G_H(\rho)\otimes \mu  
\end{equation} where $\otimes$ is the tensor product of representations. We see that $\text{Res}^G_H(\rho)\otimes \mu$ is a representation of $H$. The associativity and identity conditions are satisfied by properties of the tensor product. 
\end{proof}

\subsection{The Categories of Modules and Comodules}

Let $(A,\mu,\eta)$ denote an algebra and let $(C,\Delta, \varepsilon)$ denote a coalgebra.

\begin{defn}
The \textbfit{category of right $A$-modules}  $\mathcal{M}_A$ is the category whose objects are right $A$-modules and whose morphisms are module homomorphisms between right $A$-modules. The category of left $A$-modules is denoted ${}_A\mathcal{M}$.
\end{defn}
\begin{defn}
The \textbfit{category of right $C$-comodules}  $\mathcal{M}^C$ is the category whose objects are right $C$-comodules and whose morphisms are comodule homomorphisms between right $C$-modules. The category of left $C$-comodules is denoted ${}^C\mathcal{M}$.
\end{defn}

\begin{prop}\cite[Proposition 4.6]{joyalstreet91}\label{equivcatrep}
The category of representations of $G$, $\normalfont{\textbf{Rep}}(G)$, is equivalent to the category $\mathcal{M}^{O(G)}$.
\end{prop}

\begin{defn}\cite[p. 90]{riehl16}
A functor $F$ \textbfit{preserves and reflects} short exact sequences if  $0\rightarrow{A}\rightarrow{B}\rightarrow{C}\rightarrow{0}$ is a short exact sequence if and only if $0\rightarrow{F(A)}\rightarrow{F(B)}\rightarrow{F(C)}\rightarrow{0}$ is a short exact sequence.
\end{defn}

\begin{defn}\cite[p. 160]{mullerschneider99}
We say that a right comodule $N$ over a coalgebra $C$ is \textbfit{faithfully coflat} if the \textit{cotensor product} $N\underset{C}{\square}*:{}^C\mathcal{M}\rightarrow {}_k\mathcal{M}$ preserves and reflects short exact sequences. For $(N_1,\Delta_{N_1})\in \mathcal{M}^C$ and $(N_2,\Delta_{N_2})\in {}^C\mathcal{M}$, the cotensor product $N_1\underset{C}{\square} N_2$ is the kernel of 
\begin{equation}\Delta_{N_1}\otimes id_{N_1} -id_{N_2}\otimes \Delta_{N_2}:{N_1}\otimes N_2 \rightarrow{N_1\otimes C\otimes N_2}\end{equation}
\end{defn}

Suppose that we have a Hopf algebra $H$ with multiplication $\mu$. We have the following proposition regarding the category of right $H$-comodules.

\begin{prop}\label{catcomodules}\cite[Example 3.1.3]{christodoulou16}
$\mathcal{M}^H$ is a monoidal category. 
\end{prop}
\begin{proof}
We define the product $\otimes:\mathcal{M}^H\times \mathcal{M}^H \rightarrow{\mathcal{M}^H}$ to be the usual tensor product of vector spaces along with the following coaction. For right $H$-comodules $N_1,N_2$ with coactions $\Delta_{N_1},\Delta_{N_2}$ respectively, we define the coaction $\Delta_{N_1\otimes N_2}$ of $N_1\otimes N_2$ on $H$ to be the map
\begin{equation}\Delta_{N_1\otimes N_2}=(1_{N_1\otimes N_2}\otimes\mu )\circ(1_{N_1}\otimes \tau_{H,N_2}\otimes 1_H)\circ(\Delta_{N_1}\otimes \Delta_{N_2})=f_3\otimes f_2\otimes f_1\end{equation}
where $\tau_{H,N_2}(h\otimes n_2)=n_2\otimes h$. To illustrate this coaction,
\begin{equation}(N_1\otimes N_2)\xrightarrow{f_1}{N_1\otimes H\otimes N_2\otimes H}\xrightarrow{f_2}{N_1\otimes N_2\otimes H\otimes H}\xrightarrow{f_3}N_1\otimes N_2\otimes H\end{equation}
The associativity and identity conditions are inherited from those of the tensor product.  
\end{proof}

\begin{remark}
Similarly, the category of left $H$-modules is a monoidal category. For left $H$-modules $M_1,M_2$, the action of $M_1\otimes M_2$ on $H$ is given by
\begin{equation}h\cdot(m_1\otimes m_2)=\sum h_{(1)}\cdot m_1\otimes h_{(2)}\cdot m_2\end{equation} for $h\in H,m_1\in M_1,m_2\in M_2$.
\end{remark}

Suppose that $A$ is a right coideal subalgebra of $H$,

\begin{defn}\cite[p. 454]{takeuchi79}
Denote by $\mathcal{M}^H_A$ the category whose objects are right $A$-modules, right $H$-comodules such that $\mu_A:M\otimes A\rightarrow{M}$ is a map of $H$-comodules, and whose morphisms are right $A$-module homomorphisms, right $H$-comodule homomorphisms.
\end{defn}

\subsection{Relationship between Subgroups of Quantum Groups and Module Categories}\label{subgroupcategories}

In Proposition \ref{rephrepg}, we showed that, for a subgroup $H$ of $G$, $\textbf{\textrm{Rep}}(H)$ is a module category over the category $\textbf{\textrm{Rep}}(G)$. Since $\textbf{\textrm{Rep}}(G)$ is equivalent to $\mathcal{M}^{O(G)}$ by Proposition \ref{equivcatrep}, $\mathcal{M}^{O(H)}$ is a module category over $\mathcal{M}^{O(G)}$. We expect to find similar results for the categories of comodules over quantum groups. In \cite{christodoulou16}, Christodoulou considers the quantum version of this result for reductive subgroups.
\begin{defn}\label{reductivesubgroup}
A \textbfit{reductive subgroup} $C$ of $O_q(G)$ is a subgroup such that $O_q(G)$ is faithfully coflat over $C$.
\end{defn}

\begin{remark}
In \cite[Definition 3.3]{christodoulou16},  Christodoulou defines subgroups of $O_q(G)$ and $U_q(\mathfrak{g})$ to be what we instead call `reductive subgroups' of $O_q(G)$. We broadened her definition of a subgroup to include subgroups such as the Borel subgroups, which do not satisfy the faithful coflatness condition.
\end{remark}
We state the following results, adapted from \cite{christodoulou16}, which provide us with our categorical characterisations of quantum subgroups. We note that \cite[Proposition 6.12]{christodoulou16} can be extended to all subgroups of $O_q(G)$.

\begin{thm} \cite[Proposition 6.12]{christodoulou16}\label{modcatsub}
If $C$ is a subgroup of $O_q(G)$, then $\mathcal{M}^C$ is a module category over $\mathcal{M}^{O_q(G)}$.
\end{thm}
\begin{prop}\cite[Corollary 4.15]{christodoulou16}
If $A$ is a subgroup of $U_q(\mathfrak{g})$, then $\mathcal{M}^{U_q(\mathfrak{g})}_A$ is a module category over $\mathcal{M}^{U_q(\mathfrak{g})}$.
\end{prop}

\section{The Category of Crystals}\label{thecategoryofcrystals}
We can also use the tools of category theory to study the combinatorial structure of certain representations of quantum groups.

\subsection{Crystals}

Suppose that the Lie algebra $\mathfrak{g}$ has a set of simple roots $\Pi$ and weight lattice $\Lambda$. The \textit{simple coroots} are given by $\alpha^\vee=\frac{2}{(\alpha,\alpha)}\alpha$ for $\alpha\in\Pi$.

\begin{defn}\cite[Section 7.2]{kashiwara95}
A \textbfit{$U_q(\mathfrak{g})$-crystal} consists of a finite set $B$ equipped with maps 
\begin{align}
\begin{aligned}
\text{wt} &: B\rightarrow{\Lambda}\\
\varepsilon_\alpha,\varphi_\alpha &:B\rightarrow{\mathbb{Z}\sqcup \{-\infty\}}\\
\tilde{E}_\alpha,\tilde{F}_\alpha &:B\rightarrow{B\sqcup \{0\}}
\end{aligned}
\end{align}
for $\alpha\in\Pi$, such that, for all $b,b'\in B$, 
\begin{align}
&\varphi_\alpha(b)=\varepsilon_\alpha(b)+(\alpha^\vee,\text{wt}(b)),\\
&b'=\tilde{F}_\alpha b\text{ if and only if }b=\tilde{E}_\alpha b',\\
&\text{If }\varphi_\alpha(b)=-\infty,\text{ then }\tilde{E}_\alpha b=\tilde{F}_\alpha b=0 
\end{align}
If $\tilde{E}_\alpha b\neq 0$ then 
\begin{align}
\varepsilon_\alpha(\tilde{E}_\alpha b)&=\varepsilon_\alpha(b)-1,\quad & \varphi_\alpha(\tilde{E}_\alpha b) &=\varphi_\alpha(b)+1,\quad & \text{wt}(\tilde{E}_\alpha b) &=\text{wt}(b)+\alpha,\\
\intertext{If $\tilde{F}_\alpha b\neq 0$ then}
\varepsilon_\alpha(\tilde{F}_\alpha b)&=\varepsilon_\alpha(b)+1,\quad & \varphi_\alpha(\tilde{F}_\alpha b) &=\varphi_\alpha(b)-1,\quad & \text{wt}(\tilde{F}_\alpha b) &=\text{wt}(b)-\alpha,
\end{align}

\end{defn}

\begin{remark}
When there is no ambiguity we simply refer to $U_q(\mathfrak{g})$-crystals as crystals.
\end{remark} 
We visualise crystals using crystal graphs.

\begin{defn}\cite[Section 4.2]{kashiwara95}
A \textbfit{crystal graph} has non-zero $b\in B$ as its vertices, with arrows labelled by $\alpha\in\Pi$ such that
\begin{equation}b\xrightarrow{\alpha}{b'}\text{ if and only if } \tilde{F}_\alpha b=b'\end{equation}
 
\end{defn}

\begin{exmp}\label{tlambda}
For $\lambda\in\Lambda$, let $T_\lambda$ be the crystal with finite set $T_\lambda=\{t_\lambda\}$ along with maps
\begin{equation}\tilde{E}_\alpha(t_\lambda)=\tilde{F}_\alpha(t_\lambda)=0,\quad {\normalfont{\text{wt}}}(t_\lambda)=\lambda,\quad \varepsilon_\alpha(t_\lambda)=\varphi_\alpha(t_\lambda)=-\infty\end{equation}
The crystal graph of $T_\lambda$ consists of a single point.
\end{exmp}

\begin{defn}\label{crysmorphism}\cite[Section 7.2]{kashiwara95}
For crystals $B_1, B_2$, a \textbfit{morphism of crystals} $\psi:B_1\rightarrow{B_2}$ is a map $\psi:B_1\sqcup \{0\}\rightarrow{B_2\sqcup \{0\}}$ such that, for all $b\in B_1$ and $\alpha\in\Pi$, we have
\begin{equation}
\psi(0)=0,    
\end{equation}
For $\psi(b)\neq 0$,
\begin{equation}
\text{wt}(\psi(b))=\text{wt}(b),\qquad \varepsilon_\alpha(\psi(b))=\varepsilon_\alpha(b),\qquad  \varphi_\alpha(\psi(b))=\varphi_\alpha(b),
\end{equation}
And,
\begin{align}
\begin{aligned}
\text{If }\psi(\tilde{E}_\alpha b)&\neq 0,\,\psi(b)\neq 0\text{ then }\psi(\tilde{E}_\alpha b)=\tilde{E}_\alpha\psi (b),\\
\text{If }\psi(\tilde{F}_\alpha b)&\neq 0,\,\psi(b)\neq 0\text{ then }\psi(\tilde{F}_\alpha b)=\tilde{F}_\alpha\psi (b),\\
\end{aligned}
\end{align}

\end{defn}

\begin{defn}
We define the \textbfit{category of crystals} \textbf{Crys} to be the category with crystals as its objects and morphisms as defined above.
\end{defn}
We can also define a tensor product of crystals.

\begin{defn}\cite[Section 7.3]{kashiwara95}\label{monoidalcrystalcat}
Let $B_1,B_2$ be crystals. Their \textbfit{tensor product $B_1\otimes B_2$} is the set $\{b_1\otimes b_2\mid b_1\in B_1, b_2\in B_2\}$ with crystal structure defined by
\begin{align}
    \begin{aligned}
    \text{wt}(b_1\otimes b_2)&=\text{wt}(b_1)+\text{wt}(b_2),\\
    \varepsilon_\alpha(b_1\otimes b_2)&=\text{max}(\varepsilon_\alpha(b_1), \varepsilon_\alpha(b_2)-(\alpha^\vee,\text{wt}(b_1))),\\
    \varphi_\alpha(b_1\otimes b_2)&=\text{max}(\varphi_\alpha(b_2), \varphi_\alpha(b_1)+(\alpha^\vee,\text{wt}(b_2))),\\
    \tilde{E}_\alpha(b_1\otimes b_2)&=\begin{cases}
        \tilde{E}_\alpha b_1\otimes b_2, & \text{if } \varphi_\alpha(b_1)\geq \varepsilon_\alpha(b_2), \\b_1\otimes\tilde{E}_\alpha b_2, & \text{if } \varphi_\alpha(b_1)< \varepsilon_\alpha(b_2),
        \end{cases},\\
    \tilde{F}_\alpha(b_1\otimes b_2)&=\begin{cases}
        \tilde{F}_\alpha b_1\otimes b_2, & \text{if } \varphi_\alpha(b_1)> \varepsilon_\alpha(b_2), \\
        b_1\otimes\tilde{F}_\alpha b_2, & \text{if } \varphi_\alpha (b_1)\leq \varepsilon_\alpha(b_2).
        \end{cases}  
    \end{aligned}
\end{align}
\end{defn}

\begin{prop}\label{crysmonoidal}
The category {\normalfont{\textbf{Crys}}} has the structure of a monoidal category.\label{crys}
\end{prop}

\begin{proof}
The tensor product of crystals is a bifunctor on\\ \normalfont{\textbf{Crys}}$\otimes$\normalfont{\textbf{Crys}}. The unit object is the crystal $B(0)=\{b_0\}$ where $\tilde{E}_\alpha b_0=0=\tilde{F}_\alpha b_0$ and $\text{wt}(b_0)=0$. The associativity and identity conditions can easily be verified.
\end{proof}

\subsection{Crystal Bases}

The main source of examples of objects in \textbf{Crys} are \textbfit{crystal bases} of integrable $U_q(\mathfrak{g})$-modules $V$ in the category $\mathcal{O}_{int}^q(\mathfrak{g})$.

\begin{defn}\cite[Definition 2.2]{kashiwara95}
An \textbfit{integrable (left) $U_q(\mathfrak{g})$-module} $V$ is a type 1 module such that, for any $\alpha\in\Pi$, $V$ is a union of finite-dimensional $U_q^\alpha(\mathfrak{g})$-submodules (Definition \ref{uqasl2}). We let $\mathcal{O}_{int}^q(\mathfrak{g})$ be the category consisting of integrable $U_q(\mathfrak{g})$-modules with ${\normalfont{\text{dim }}}U_q^+(\mathfrak{g})u <\infty$ for all $u\in V$.
\end{defn}

The definition of a crystal base is too involved to state here, for more information see Kashiwara's survey text \cite{kashiwara95}. Kashiwara's crystal bases can be understood as bases at $q=0$ and can be extended to global bases covering all values of $q$. Independently, Lusztig \cite{lusztig90ii} introduced canonical bases which coincide with these global bases.

We note that a crystal base consists of a pair $(L,B)$ with $L$ a `crystal lattice' and $B$ a basis of $L/qL$. $B$ is acted on by maps $\tilde{E}_\alpha,\tilde{F}_\alpha$ which leave each $U_q^\alpha(\mathfrak{g})$-submodule invariant. Furthermore, each $b\in B$ has an associated weight $\text{wt}(b)$. Therefore a crystal base $(L,B)$ gives a crystal $B$ in \textbf{Crys}.

We state the following proposition regarding modules in $\mathcal{O}_{int}^q(\mathfrak{g})$.

\begin{prop}\label{irreduciblereps}\cite[Theorem 2.1]{kashiwara95}
Modules in $\mathcal{O}_{int}^q(\mathfrak{g})$ are semisimple, and all simple modules  are isomorphic to $V(\lambda)$, for some $\lambda\in\Lambda^+$, where $V(\lambda)$ is the $U_q(\mathfrak{g})$-module generated by a highest weight vector $u_\lambda$ with the following relations
\begin{equation}
E_\alpha u_\lambda=0=F_\alpha^{1+a_{\alpha\lambda}}u_\lambda, \quad K_\mu u_\lambda=q^{(\lambda,\mu)}u_\lambda    
\end{equation}
For all $\alpha\in\Pi,\quad \mu\in\mathbb{Z}\Phi$.
\end{prop}

These simple modules have crystal bases.

\begin{thm}\cite[Theorem 4.2]{kashiwara95}\cite[Section 9.5]{jantzen96}\label{crystalv(lambda)}
$V(\lambda)$ has crystal lattice $L(\lambda)$ equal to the $\textbf{A}$-submodule of $V(\lambda)$ spanned by vectors of the form \begin{equation}
\tilde{F}_{\alpha_1}\tilde{F}_{\alpha_2}\dots \tilde{F}_{\alpha_k}u_\lambda\text{  for } \alpha_i\in\Pi,k\geq 0,   
\end{equation}where $\textbf{A}=\{g/h\mid g,h\in k[q],\, h(0)\neq0\}$. The associated crystal $B(\lambda)$ is the set 
\begin{equation}B(\lambda)=\{\tilde{F}_{\alpha_1}\tilde{F}_{\alpha_2}\dots \tilde{F}_{\alpha_k}u_\lambda{\normalfont{\text{ mod }}}qL(\lambda) \mid \alpha_i\in \Pi, k\geq 0\}\backslash \{0\}\end{equation} 
\end{thm}

Further, we can define crystal bases for all modules in $\mathcal{O}_{int}^q(\mathfrak{g})$.

\begin{thm}\cite[Theorem 4.3]{kashiwara95}\label{reducibility}
Suppose $V$ is an $U_q(\mathfrak{g})$-module in $\mathcal{O}_{int}^q(\mathfrak{g})$ with $V\simeq\oplus_{\lambda\in\Lambda^+}V(\lambda)$. If $(L,B)$ is a crystal basis of $V$ then there exists an isomorphism \begin{equation}
(L,B)\simeq \bigoplus_{\lambda\in\Lambda^+} (L(\lambda),B(\lambda))    
\end{equation}
\end{thm}

\subsection{Crystal Bases for $U_q(\mathfrak{sl}_2)$-modules in the Category $\mathcal{O}_{int}^q(\mathfrak{sl}_2)$}\label{sl2crys}

We apply the definitions and theorems stated in the previous section to $U_q(\mathfrak{sl}_2)$-modules in the category $\mathcal{O}_{int}^q(\mathfrak{sl}_2)$. From Theorems \ref{irreducibles} and \ref{generate}, every simple $U_q(\mathfrak{sl}_2)$-module is isomorphic to $V(n)$ for some $n\in\mathbb{Z}_{\geq 0}$. Therefore, we just need to consider crystal bases for these $V(n)$ by Theorem \ref{reducibility}. These modules have bases $\{u,Fu,\dots F^{(n)}u\}$. Define operators $\tilde{E},\tilde{F}$ on the basis by
\begin{equation}
    \tilde{E}(F^{(k)}u)=F^{(k-1)}u,\quad \tilde{F}(F^{(k)})u=F^{(k+1)}u
\end{equation}
Thus, the crystal lattice $L(n)$ and crystal $B(n)$ are
\begin{equation}
    L(n)=\bigoplus_{k=0}^n\textbf{A}F^{(k)}u,\quad B(n)=\{\overline{u},\overline{Fu},\dots,\overline{F^{(n)}u}\}
\end{equation}
where $\overline{F^{(k)}u}=F^{(k)}u \text{ mod }qL(n)$. The maps giving $B(n)$ its crystal structure are given by 
\begin{equation}
    {\normalfont{\text{wt}}}(\overline{F^{(k)}u})=n-2k,\quad \varepsilon(\overline{F^{(k)}u})=k,\quad \varphi(\overline{F^{(k)}u})=n-k
\end{equation}
And, the crystal graph of $B(n)$ is given by 
\begin{equation}
    \overline{u}\rightarrow{\overline{Fu}}\rightarrow{\dots}\rightarrow{\overline{F^{(n)}u}}
\end{equation}
Crystal graphs provide us with a way of visualising the tensor product of crystals, as the following example shows. Consider the tensor product $B(2)\otimes B(2)$, where $B(2)$ has basis $\{\overline{u},\overline{Fu},\overline{F^{(2)}u}\}$; the crystal graph of $B(2)\otimes B(2)$ is

\begin{center}
\begin{tikzpicture}[scale =2.5]
\node(a) at (0,0.6) {$\boldsymbol{\overline{u}}$};
\node(b) at (1,0.6) {$\boldsymbol{\overline{Fu}}$};
\node(c) at (2,0.6) {$\boldsymbol{\overline{F^{(2)}u}}$};
\node(d) at (-0.4,0.6) {$\boldsymbol{B(2)}$};
\draw[->] (a) -- (b);
\draw[->] (b)--(c);

\node (e) at (-0.8, 0) {$\boldsymbol{\overline{u}}$};
\node (f) at (-0.8, -0.5) {$\boldsymbol{\overline{Fu}}$};
\node (g) at (-0.8, -1) {$\boldsymbol{\overline{F^{(2)}u}}$};
\node (h) at (-0.8,0.3) {$\boldsymbol{B(2)}$};
\draw[->] (e) -- (f);
\draw[->] (f) -- (g);

\node (A) at (0,0) {$\overline{u}\otimes \overline{u}$};
\node (B) at (1,0) {$\overline{Fu}\otimes \overline{u}$};
\node (C) at (2,0) {$\overline{F^{(2)}u}\otimes \overline{u}$};
\node (D) at (0,-0.5) {$\overline{u}\otimes \overline{Fu}$};
\node (E) at (1,-0.5) {$\overline{Fu}\otimes \overline{Fu}$};
\node (F) at (2,-0.5) {$\overline{F^{(2)}u}\otimes \overline{Fu}$};
\node (G) at (0,-1) {$\overline{u}\otimes \overline{F^{(2)}u}$} ;
\node (H) at (1,-1) {$\overline{Fu}\otimes \overline{F^{(2)}u}$} ;
\node (I) at (2,-1) {$\overline{F^{(2)}u}\otimes \overline{F^{(2)}u}$} ;
\draw[->] (A) -- (B);
\draw[->] (B) -- (C);
\draw[->] (C) -- (F);
\draw[->] (F) -- (I);
\draw[->] (D) -- (E);
\draw[->] (E) -- (H);
\end{tikzpicture}
\end{center}

In general, the crystal graph of $B(m)\otimes B(n)$ is

\begin{center}
\begin{tikzpicture}[scale= 0.9]
\node (A) at (-1,1.5) {$\boldsymbol{B(m)}$};
\node (a) at (0,1.5) {$\boldsymbol{\circ}$};
\node (b) at (1,1.5) {$\boldsymbol{\circ}$};
\node (c) at (2,1.5) {$\boldsymbol{\circ}$};
\node (d) at (3,1.5) {$\boldsymbol{\circ}$};
\node (e) at (4,1.5) {$\boldsymbol{\circ}$};
\node (f) at (5,1.5) {$\boldsymbol{\circ}$};
\node (g) at (6,1.5) {$\boldsymbol{\circ}$};

\draw[->] (a) -- (b);
\draw[->] (b) -- (c);
\draw[->] (c) -- (d);
\draw[dashed] (d) -- (e);
\draw[->] (e) -- (f);
\draw[->] (f) -- (g);

\node (B) at (-1.5,0.7) {$\boldsymbol{B(n)}$};
\node (h) at (-1.5,0) {$\boldsymbol{\circ}$};
\node (i) at (-1.5,-1) {$\boldsymbol{\circ}$};
\node (j) at (-1.5,-2) {$\boldsymbol{\circ}$};
\node (k) at (-1.5,-3) {$\boldsymbol{\circ}$};
\node (l) at (-1.5,-4) {$\boldsymbol{\circ}$};

\draw[->] (h) -- (i);
\draw[dashed] (i) -- (j);
\draw[dashed] (j) -- (k);
\draw[dashed] (k) -- (l);

\node (1) at (0,0) {$\circ$};
\node (2) at (1,0) {$\circ$};
\node (3) at (2,0) {$\circ$};
\node (4) at (3,0) {$\circ$};
\node (5) at (4,0) {$\circ$};
\node (6) at (5,0) {$\circ$};
\node (7) at (6,0) {$\circ$};

\draw[->] (1) -- (2);
\draw[->] (2) -- (3);
\draw[->] (3) -- (4);
\draw[dashed] (4) -- (5);
\draw[->] (5) -- (6);
\draw[->] (6) -- (7);

\node (8) at (0,-1) {$\circ$};
\node (9) at (1,-1) {$\circ$};
\node (10) at (2,-1) {$\circ$};
\node (11) at (3,-1) {$\circ$};
\node (12) at (4,-1) {$\circ$};
\node (13) at (5,-1) {$\circ$};
\node (14) at (6,-1) {$\circ$};

\draw[->] (8) -- (9);
\draw[->] (9) -- (10);
\draw[->] (10) -- (11);
\draw[dashed] (11) -- (12);
\draw[->] (12) -- (13);

\node (17) at (5,-2) {$\circ$};
\node (18) at (6,-2) {$\circ$};

\node (21) at (5,-3) {$\circ$};
\node (22) at (6,-3) {$\circ$};

\node (23) at (0,-4) {$\circ$};
\node (24) at (1,-4) {$\circ$};
\node (25) at (2,-4) {};

\node (27) at (5,-4) {$\circ$};
\node (28) at (6,-4) {$\circ$};

\draw [->] (23) -- (24);
\draw[dashed] (24) -- (25);

\draw[->] (7)-- (14);
\draw[->] (14)-- (18);
\draw[->] (18)-- (22);
\draw[->] (22)-- (28);
\draw[->] (13)-- (17);
\draw[->] (17)-- (21);
\draw[->] (21)-- (27);

\draw[dashed] (0,-2) -- (4,-2);
\draw[dashed] (0,-3) -- (3,-3);
\draw[dashed] (4,-2) -- (4,-4);
\draw[dashed] (3,-3) -- (3,-4);
\end{tikzpicture}
\end{center}

\subsection{Crystal Bases for $O_q(G)$}
We can translate our crystal basis theory from $U_q(\mathfrak{g})$ to $O_q(G)$. 

\begin{lemma}\cite[Section 0.1]{kashiwara93}\label{oqgcrystal}
Recall the Quantum Peter-Weyl Theorem for $O_q(G)$ given in \ref{qpw}. The corresponding crystal basis of $O_q(G)$ is

\begin{equation}
    (L(O_q(G)),B(O_q(G)))=\bigoplus_{\lambda\in\Lambda^+}(L(\lambda), B(\lambda))\otimes (L(\lambda)^*,B(\lambda)^*)
\end{equation}
So, the crystal associated with $O_q(G)$ is
\begin{equation}
    B(O_q(G))=\bigsqcup_{\lambda\in\Lambda^+}B(\lambda)\otimes B(-\lambda)
\end{equation}
where $B(-\lambda)$ is defined to be $B(\lambda)^\vee\simeq B(\lambda)^*$ , where $B(\lambda)^\vee=\{b^\vee\mid b\in B(\lambda)\}$ with $\tilde{E}_\alpha(b^\vee)=(\tilde{F}_\alpha(b))^\vee$, $\tilde{F}_\alpha(b^\vee)=(\tilde{E}_\alpha(b))^\vee$ and ${\normalfont{wt}}(b^\vee)=-{\normalfont{wt}}(b)$. 
\end{lemma}
\begin{remark}
We see that $B(\lambda)^\vee$ is the crystal obtained from $B(\lambda)$ by reversing the direction of arrows in the crystal graph.
\end{remark}
\begin{exmp}
Therefore, in the case of $O_q(SL_2)$, our corresponding crystal is
\begin{equation}B(O_q(SL_2))=\bigsqcup_{n\in\mathbb{Z}_{\geq 0}}B(n)\otimes B(-n)\end{equation}
However, we see from \ref{sl2crys} that $B(n)\simeq B(-n)$, and so 
\begin{equation}B\simeq\bigsqcup_{n\in\mathbb{Z}_{\geq 0}}B(n)^{\otimes 2}\end{equation}

\end{exmp}

\section{Crystal Bases for Subgroups of Quantum Groups}\label{crystalbasesforsubgroupsofquantumgroups}
Extending crystal basis theory to subgroups of quantum groups is an open and under-researched problem. Currently, only a handful of subgroups can be given crystal bases, including the case of a symmetric pair \cite{watanabe18,huanchenwang18}. Here we discuss constructions of crystal bases for certain subgroups of quantum groups and remark on why we cannot use similar methods for all subgroups.  

\subsection{The Crystal Bases of $U_q^\pm(\mathfrak{g})$}

We consider $U_q(\mathfrak{g})$ as a module over itself via the adjoint action and obtain the following weight space decomposition \cite[Definition I.8.11]{browngoodearl02}
\begin{equation}
    U_q(\mathfrak{g})=\bigoplus_{\xi\in\mathbb{Z}\Phi}U_q(\mathfrak{g})_\xi, \text{ where } U_q(\mathfrak{g})_\xi=\{ u\in U_q(\mathfrak{g})\mid K_\mu u K_\mu^{-1}=q^{(\xi,\mu)}u\text{ for }\mu\in\mathbb{Z}\Phi\}
\end{equation}
Moreover, this is a $\mathbb{Z}\Phi$-\textit{graded algebra} since $U_q(\mathfrak{g})_\xi U_q(\mathfrak{g})_\eta\subseteq U_q(\mathfrak{g})_{\xi+\eta}$. We notice that $K_\alpha^{\pm 1}\in U_q(\mathfrak{g})_0$, $E_\alpha\in U_q(\mathfrak{g})_\alpha$ and $F_\alpha\in U_q(\mathfrak{g})_{-\alpha}$ for all $\alpha\in\Pi$.
Therefore, we can restrict this grading to give $\mathbb{Z}^{\geq 0}\Phi$-gradings on the subgroups $U_q^{\geq 0}$, $U_q^+(\mathfrak{g})$ and $\mathbb{Z}^{\leq 0}\Phi$-gradings on the subgroups
$U_q^{\leq 0}$, $U_q^-(\mathfrak{g})$.\\

In \cite[Section 8]{kashiwara95}, M. Kashiwara discusses a construction of a crystal base of $U_q^-(\mathfrak{g})$, derived from the crystal base of $V(\lambda)$ as $\lambda$ tends to infinity. For $\lambda\in\Lambda^+$, consider the surjective $U_q^-(\mathfrak{g})$ homomorphism
\begin{equation}
\pi_\lambda:U_q^-(\mathfrak{g})\rightarrow{V(\lambda)},\text{  given by }\pi_\lambda(x)=x u_\lambda
\end{equation}
where $u_\lambda$ is the highest weight vector of $V(\lambda)$. $U_q^-(\mathfrak{g})$ has a $\mathbb{Z}_{\leq 0}$-grading given by 
\begin{equation}
U_q^-(\mathfrak{g})=\bigoplus_{\xi\in\mathbb{Z}_{\leq 0}\Phi}U_q(\mathfrak{g})_\xi    
\end{equation}
Suppose that $x\in U_q^-(\mathfrak{g})_\xi$. Then for all $\alpha\in\Pi$ we have 
\begin{align}
\begin{aligned}
K_\mu\pi_\lambda(x) &=K_\mu xu_\lambda,\\
                    &=q^{(\xi,\mu)}x K_\mu u_\lambda, \text{ since } K_\mu x K_\mu^{-1}=q^{(\xi,\mu)}x\text{ as } x\in U_q^-(\mathfrak{g})_\xi,\\
                    &=q^{(\xi,\mu)}xq^{(\lambda,\mu)}u_{\lambda},\text{ since } u_\lambda\in V(\lambda), \\
                    & =q^{(\xi+\lambda, \mu)} \pi_\lambda(x)
\end{aligned}
\end{align}

Therefore,
\begin{equation}
\pi_\lambda(U_q^-(\mathfrak{g})_\xi)\subseteq \{u\in V(\lambda):K_\mu u=q^{(\lambda+\xi,\mu)}u\text{ for all } \mu\in\mathbb{Z}\Phi \}=V(\lambda)_{\lambda+\xi}   
\end{equation}
From Definition \ref{irreduciblereps}, we see that the kernel of $\pi_\lambda$ is $\sum_{\alpha\in\Pi}U_q^-(\mathfrak{g})F_\alpha^{1+a_{\lambda\alpha}}$, as $u_\lambda$ is a highest weight vector and $U_q^-(\mathfrak{g})$ is spanned by $F_\alpha$ for $\alpha\in\Pi$. If $a_{\lambda\alpha}\gg 0$, then ${\normalfont{\text{Ker }}}\pi_\lambda\cap U_q^-(\mathfrak{g})_\xi=\{0\}$ so that, by the first isomorphism theorem
\begin{equation}
    U_q^-(\mathfrak{g})_\xi\simeq V(\lambda)_{\lambda+\xi}
\end{equation}
Therefore, using the grading of $U_q^-(\mathfrak{g})$,  we can regard $U_q^-(\mathfrak{g})$ as the limit of $V(\lambda)$ as $a_{\lambda\alpha}$ tends to infinity. \\

We now examine the behaviour of $V(\lambda)$ as $\lambda$ tends to infinity. Kashiwara shows that, for $\eta\in\Lambda^+$, there exists a  map $p_{\lambda,\lambda+\eta}:V(\lambda+\eta)\rightarrow{V(\lambda)}$ sending $L(\lambda+\eta)$ to $L(\lambda)$. This induces a morphism of crystals $B(\lambda+\eta)\rightarrow{B(\lambda)\otimes B(\eta)}\rightarrow{B(\lambda)\otimes T_\eta}$, where $T_\eta$ is defined in \ref{tlambda}. This morphism commutes with the $F_\alpha s$ but not the $E_\alpha s$. Therefore, as $\lambda$ tends to infinity, we obtain a crystal basis $(L(\infty),B(\infty))$ of $U^-_q(\mathfrak{g})$.

Explicitly, $L(\infty)$ is the $\textbf{A}$-submodule of $U_q^-(\mathfrak{g})$ spanned by vectors of the form $\tilde{F}_{\alpha_1}\tilde{F}_{\alpha_2}\dots \tilde{F}_{\alpha_k}\cdot 1$ $(k\geq 0, \alpha_i\in\Pi)$ and $B(\infty)$ is the set
\begin{equation}
B(\infty)=\{\tilde{F}_{\alpha_1}\tilde{F}_{\alpha_2}\dots \tilde{F}_{\alpha_k}\cdot 1{\normalfont{\text{ mod }}qL(\infty)}\mid \alpha_i\in\Pi, k\geq 0\}\backslash\{0\}   
\end{equation}
We obtain a similar basis $(L(-\infty),B(-\infty))$ of $U_q^+(\mathfrak{g})$

\begin{remark}
We see that the method used in this construction cannot generally be extended to other subgroups of $U_q(\mathfrak{g})$ since they don't always have the same graded structure. As such, we cannot always find appropriate isomorphisms with our simple modules $V(\lambda)$.
\end{remark}

\subsection{The Crystal Bases of the Subalgebra of Invariants of $U_q^\pm(\mathfrak{g})$ on $O_q(G)$}

We now consider the $O_q(G)$ case. Recall the quantum Peter-Weyl Theorem from Theorem \ref{qpw}. We examine the actions of $U_q^\pm(\mathfrak{g})$ on this decomposition.
\begin{thm}\cite[Lemma 4.15]{ganev16}\cite[Theorem 7.8]{letzter02}
The set of right invariants of $U^+_q(\mathfrak{g})$ on $O_q(G)$ can be expressed using the quantum Peter-Weyl Theorem as
\begin{equation}A=O_q(G)^{U_q^+(\mathfrak{g})}=\bigoplus_{\lambda\in\Lambda^+}V(\lambda)\otimes (V(\lambda)^*)^{U_q^+(\mathfrak{g})}\end{equation}
\end{thm}

Further, the space $(V(\lambda)^*)^{U_q^+(\mathfrak{g})}$ is the one-dimensional $U_q^0(\mathfrak{g})$-module generated by the dual of the highest weight vector of $V(\lambda)$, $u_\lambda^*$ \cite[Lemma 4.15]{ganev16}. Therefore, 
\begin{equation}
V(\lambda)\otimes (V(\lambda)^*)^{U_q^+(\mathfrak{g})}\simeq V(\lambda)
\end{equation}
And so, 
\begin{equation}
    A\simeq\bigoplus_{\lambda\in\Lambda^+}V(\lambda)
\end{equation}
Similarly, the set of left invariants of $U_q^+(\mathfrak{g})$ on $O_q(G)$ can be expressed as 

\begin{equation}
{}^{U_q^+(\mathfrak{g})} O_q(G)\simeq\bigoplus_{\lambda\in\Lambda^+}V(\lambda)^*
\end{equation}

We obtain a similar result by considering the invariants of the action of $U_q^-(\mathfrak{g})$.
$A$ is a right coideal subalgebra of $O_q(G)$ by Proposition \ref{prop1} and so $O_q(G)/O_q(G)A^+$ is a subgroup of $O_q(G)$ by Proposition \ref{prop2}. We can use Theorem \ref{crystalv(lambda)} to find the crystal base associated with $A$ and we find that the associated crystal is 
\begin{equation}B(A)=\bigsqcup_{\lambda\in\Lambda^+}B(\lambda)\end{equation}

\begin{remark}
However, we cannot find the crystals associated with all coideal subalgebras of $O_q(G)$ constructed as sets of invariants in this way. This is because the right-hand factor may consist of a module generated by a set of polynomials in the duals of the highest weight vectors, to which we cannot assign crystals. Moreover, we cannot find crystal bases for our quotient left $O_q(G)$-module coalgebras, equivalently subgroups of $O_q(G)$, using our current techniques. 
\end{remark}

\section{Topics for Further Research}\label{furtherresearch}
We consider a number of possible extensions of material in this essay.

\subsection{Subgroups of $U_q(\mathfrak{sl}_2)$ and $O_q(SL_2)$}

We saw in Section \ref{subgroupsuq(sl2)} that all right coideal subalgebras of $U_q(\mathfrak{sl}_2)$ are of the stated form. We constructed an example of a subgroup of $O_q(G)$ using one of these subgroups and Propositions \ref{prop1}, \ref{prop2}. However, this is quite a laborious process. We could perhaps apply similar reasoning to Vocke in \cite{vocke18} to find and classify the subgroups of $O_q(SL_2)$. We guess that there are fewer subgroups of $O_q(SL_2)$ than of $U_q(\mathfrak{sl}_2)$. Our intuition comes from the classical case, where there are fewer algebraic subgroups than semisimple Lie subalgebras.  
\subsection{Reductive Subgroups}

In Definition \ref{reductivesubgroup}, we defined `reductive subgroups' of $O_q(G)$ as subgroups satisfying certain coflatness conditions, using remarks in \cite{christodoulou16}. We would like to examine whether these subgroups correspond to classical reductive subgroups of $G$.

\subsection{A Categorical Definition of Subgroups of Quantum Groups}
In Section \ref{subgroupcategories}, we gave categorical characterisations of subgroups of $O_q(G)$ and $U_q(\mathfrak{g})$. It would be interesting to see under which conditions we can find a converse statement. We can then use our characterisation as a definition for a certain class of subgoups.\\

For example, in \cite[Theorem 6.11]{christodoulou16}, Christodoulou shows that if $\mathcal{M}^C$ is a module category over $\mathcal{M}^{O_q(G)}$ then, if the functor $\Psi:\mathcal{M}^{O_q(G)}\rightarrow{\mathcal{M}^C}$ satisfies certain conditions, which are too involved to state here,  then $C$ has the structure of a left $O_q(G)$-module quotient coalgebra. Moreover, $O_q(G)$ is faithfully coflat over $C$. Therefore, we should be able to obtain a restricted class of reductive subgroups of $O_q(G)$ such that we can find a converse to our categorical characterisation.\\

In \cite{christodoulou16}, Christodoulou defines subgroups of all quantum groups as quotient left module coalgebras. We changed her definition of subgroups of $U_q(\mathfrak{g})$ to emphasise the duality between $U_q(\mathfrak{g})$ and $O_q(G)$. Our categorical characterisation of subgroups $A$ of $U_q(\mathfrak{g})$ in terms of the category $\mathcal{M}^{U_q(\mathfrak{g})}_A$ does not seem restrictive enough. Therefore, we could try to develop a more restricted categorical characterisation for which we could find a converse; perhaps for a certain class of subgroups $A$ with $U_q(\mathfrak{g})$ faithfully coflat over $U_q(\mathfrak{g})$, a `reductive subgroup' of $U_q(\mathfrak{g})$. 

\subsection{Crystal Bases for Subgroups of Quantum Groups}
In Section \ref{thecategoryofcrystals}, we constructed crystal bases for integrable $U_q(\mathfrak{g})$-modules and for $O_q(G)$. In Section \ref{crystalbasesforsubgroupsofquantumgroups}, we discussed constructions of crystal bases for certain subgroups of $U_q(\mathfrak{g})$ and for specific invariants of subgroups of $U_q(\mathfrak{g})$ on $O_q(G)$, using Kashiwara's crystal bases for the simple $U_q(\mathfrak{g})$-modules in Definition \ref{crystalv(lambda)}. However, we see that our current understanding of crystal bases is limited, and that we need to develop new techniques to study crystal bases for our subgroups of $U_q(\mathfrak{g})$ and $O_q(G)$.\\

One possible solution could involve category theory. In Section \ref{subgroupcategories}, we discussed the connections between subgroups and module categories, and in Proposition \ref{crys}, we showed that the category of crystals, \textbf{Crys}, has a monoidal structure. Therefore, it seems intuitive to examine certain module categories over the category of crystals and see if they correspond to module categories associated with subgroups. We could call these `categories of module crystals'. \\

The comodules of $O_q(G)$ are precisely the $U_q(\mathfrak{g})$-modules in $\mathcal{O}_{int}^q(\mathfrak{g})$. This is because $O_q(G)$ is a direct sum of coalgebras $V(\lambda)\otimes V(\lambda)^*$, whose comodules are precisely direct sums of copies of $V(\lambda)$. We have a good understanding of the crystal bases $B(\lambda)$, as defined in Definition \ref{crystalv(lambda)}, of these $V(\lambda)$. Denote by $\textbf{Crys}_\mathfrak{g}$ the subcategory of $\textbf{Crys}$ whose objects are crystals which are coproducts 
of crystals of the form $B(\lambda)$. Suppose we have a module category $\mathcal{M}^C$ over $\mathcal{M}^{O_q(G)}$ such that the functor $\Psi:\mathcal{M}^{O_q(G)}\rightarrow{\mathcal{M}^C}$ satisfies the conditions of \cite[Theorem 6.11]{christodoulou16}, then $C$ is a reductive subgroup of $O_q(G)$. It would be interesting to try to construct a corresponding module category over $\textbf{Crys}_\mathfrak{g}$ such that the functor between them satisfies the same conditions. We could then consider this the category of module crystals corresponding to the subgroup $C$.

\section{Conclusion}
In the same way that quantum mechanics has helped shape Physicists' understanding of classical mechanics, we have seen how the study of quantum groups has informed our understanding of classical algebra. We have shown how techniques from other areas of Mathematics, such as combinatorics and category theory, can help us develop this beautiful theory. It will be exciting to see whether the ideas  discussed in this essay will provide us with the solution to the problem of finding crystal bases for subgroups of quantum groups.

\section*{Acknowledgements}
I would like to thank Professor Kobi Kremnitzer for his invaluable advice and support during the writing of this essay. 

\pagebreak
\nocite{joseph95}
\nocite{smith15}
\nocite{etingof15}
\nocite{kirillov02}
\nocite{ritter19}
\nocite{montgomery93}
\nocite{kashiwara91}
\bibliographystyle{plain}
\bibliography{references}

\end{document}